\documentclass[12pt,index]{amsart}

\usepackage{amssymb} 
\usepackage{mathtools}

\usepackage{enumerate,xcolor,mathrsfs}
\usepackage[pagebackref,colorlinks,linkcolor=red,citecolor=blue,urlcolor=blue,hypertexnames=true]{hyperref}
\usepackage{url}

\usepackage{comment}

\includecomment{fornow}\excludecomment{maybelater}

\newcommand{\df}[1]{{\it{#1}}{\index{#1}}}

\makeatletter
\newsavebox\myboxA
\newsavebox\myboxB
\newlength\mylenA

\newcommand*\xoverline[2][0.75]{%
    \sbox{\myboxA}{$\m@th#2$}%
    \setbox\myboxB\null
    \ht\myboxB=\ht\myboxA%
    \dp\myboxB=\dp\myboxA%
    \wd\myboxB=#1\wd\myboxA
    \sbox\myboxB{$\m@th\overline{\copy\myboxB}$}
    \setlength\mylenA{\the\wd\myboxA}
    \addtolength\mylenA{-\the\wd\myboxB}%
    \ifdim\wd\myboxB<\wd\myboxA%
       \rlap{\hskip 0.5\mylenA\usebox\myboxB}{\usebox\myboxA}%
    \else
        \hskip -0.5\mylenA\rlap{\usebox\myboxA}{\hskip 0.5\mylenA\usebox\myboxB}%
    \fi}
\makeatother

\newtheorem{theorem}            {Theorem}[section]
\newtheorem{corollary}          [theorem]{Corollary}
\newtheorem{proposition}        [theorem]{Proposition}

\newtheorem{lemma}              [theorem]{Lemma}
\newtheorem{remark}         [theorem]{Remark}

\newcommand{\C}{\mathbb{C}}

\newcommand{\sW}{\mathscr{W}}

\newcommand{\mfE}{\mathfrak{E}}

\newcommand{\gv}{{\tt{g}}}
\newcommand{\bcdot}{{\boldsymbol{\cdot}}}

\newcommand{\cN}{\mathcal{N}}

\newcommand{\cS}{\mathcal{S}}

\newcommand{\vg}{{\tt{g}}}
\newcommand{\vh}{{\tt{h}}}

\newcommand{\fP}{\mathfrak{P}}

\makeatletter
\def\moverlay{\mathpalette\mov@rlay}
\def\mov@rlay#1#2{\leavevmode\vtop{
		\baselineskip\z@skip \lineskiplimit-\maxdimen
		\ialign{\hfil$#1##$\hfil\cr#2\crcr}}}
\makeatother

\newcommand{\CC}{\mathbb{C}}
\newcommand{\TT}{\mathbb{T}}
\newcommand{\RR}{\mathbb{R}}

\newcommand{\fm}{\mathfrak{m}}
\newcommand{\fhm}{\fm}
\newcommand{\cT}{\mathcal{T}}
\newcommand{\fN}{\mathfrak{N}}
\newcommand{\fZ}{\mathfrak{Z}}

\newcommand{\wtvphi}{\widetilde{\varphi}}
\newcommand{\wtb}{\widetilde{b}}
\newcommand{\wtpi}{\widetilde{\pi}}

\newcommand{\fg}{\mathfrak{g}}

\newcommand{\au}{u}

\newcommand{\cF}{\mathcal{F}}

\newcommand{\cG}{\mathcal{F}}

\newcommand{\mfI}{\mathfrak{I}}

\newcommand{\derL}{\mathscr{L}}

\newcommand{\ccG}{\widetilde{\cG}}

\newcommand{\fh}{\mathfrak{h}}
\newcommand{\fH}{\mathfrak{H}}
\newcommand{\fI}{\mathfrak{I}}

\newcommand{\cI}{\mathcal I}

\newcommand{\norml}{normalized }
\newcommand{\Norml}{Normalized }

\newcommand{\AST}{\circ}

\newcommand{\qT}{{\tt{T}}}
\newcommand{\dB}{Q}
\newcommand{\newpsi}{\phi}
\newcommand{\ee}{{\tt{e}}}

\makeatletter
\newcommand{\addresseshere}{%
  \enddoc@text\let\enddoc@text\relax
}
\makeatother

\makeindex

\title[Hyper-Reinhardt Spectrahedra]{Automorphisms of Hyper-Reinhardt Free Spectrahedra}

\author[McCullough]{Scott McCullough}
\address{Department of Mathematics\\
  University of Flordia \\ Gainesville, FL}
\email{sam@ufl.edu}
\thanks{Research Supported by NSF grant DMS-1764231}

\subjclass[2010]{47L25, 32H02 (Primary); 52A05, 46L07 (Secondary)}
\keywords{bianalytic map, birational map, free spectrahedron, free analysis}

 \numberwithin{equation}{section}

\begin{document}

\begin{abstract}
 The free automorphisms of a class of Reinhardt free spectrahedra
 are trivial.
\end{abstract}

\maketitle

\section{Introduction}
 \subsection{Prologue} 
 An  overarching  conjecture,
 verified in a rather generic sense, is that  automorphisms
 of free spectrahedra 
 have a very specific transfer function like realization 
  that is highly algebraic in nature.  See, for instance, \cite{bestmaps}.
 A natural class of free spectrahedra not covered by prior results are 
 those with circular symmetry. Here we verify the conjecture
 for those free spectrahedra that satisfy a strong Reinhardt condition. 
 The approach is new.  It hinges on considering the subgroup
 of the automorphism group generated by a given automorphisms
 and those inherent in the Reinhardt condition.
 In the remainder of this introduction we state the main result and then 
 expand somewhat on the context and prior results,
  after
 first introducing the needed background material.

\subsection{Free spectrahedra}
 Fix a positive integer $\vg.$
 For positive integers $n,$ let \df{$M_n(\C)$} denote the set of $n\times n$ matrices
 and \df{$M_n(\C)^{\vg}$} the set of $\vg$-tuples $X=(X_1,\dots,X_{\vg})$ with entries
 $X_j$ from  $M_n(\C).$  Let $M(\C)^{\vg}$ denote the sequence $(M_n(\CC)^{\vg})_n.$

 Given a positive integer $d$ and a tuple $B\in M_d(\CC)^\vg,$  let
  \df{$\Lambda_B(x)$} denote the  linear (matrix) polynomial,
\[
 \Lambda_B(x) =\sum_{j=1}^{\vg} B_j x_j.
\]
  The polynomial $\Lambda_B$ naturally \df{evaluates} at 
 an $X\in M(\C)^{\vg}$ using the (Kronecker) tensor product as
\[
 \Lambda_B(X)= \sum_{j=1}^{\vg} B_j\otimes X_j.
\]
  The tensor product of  an $s\times s$ matrix $S$ and a $t\times t$
 matrix $T$ can also be interpreted as follows. On the algebraic
 tensor product $H=\CC^s\otimes \CC^t,$ the formula
\[
 \langle \sum x_j\otimes g_j, \sum y_k\otimes h_k\rangle
 =\sum_{j,k} \langle x_j,y_k\rangle \, \langle g_j,h_k\rangle
\]
 defines an inner product on $H$ and we view $H$ as the resulting
 ($st$-dimensional) Hilbert space. The tensor product of
 $S$ and $T$ is determined by linearity and
\[
  (S\otimes T) \, (x\otimes h) = Sh \otimes Th.
\]
 From this description of the tensor product it follows
 readily that $(S\otimes T)^*= S^*\otimes T^*.$
 Thus,
\[
 \Lambda_B(X)^* =\Lambda_{B^*}(X^*)=\sum_{j=1}^\vg B_j^*\otimes X_j^*.
\]
 Define $L_B$ by  \index{$L_B(X)$}
\[
 L_B(X) = I_d\otimes I +  \Lambda_B(X)  + \Lambda_B(X)^*.
\]
 From the discussion above, it follows that $L_B(X)$ is
 selfadjoint:  $L_B(X)^* =L_B(X).$ We refer the reader to  \cite{HJ}
 for further information
 about the tensor product of matrices.

For a square matrix $T$ the notation \df{$T\succ 0$} (resp. \df{$T\succeq 0$})
 indicates $T$ is positive definite (resp. positive semidefinite). 
 Let, for positive integers $n,$ \index{$\fP_B[n]$}
\[
 \fP_B[n] =\{X\in M_n(\C)^{\vg}: L_B(X)\succ 0\}.
\]
 The set $\fP_B[1]$ is a \df{spectrahedron} and
 the sequence of sets  $\fP_B =\left (\fP_B[n] \right)_{n=1}^\infty$ \index{$\fP_B$}
 is  a \df{free spectrahedron}. 
 As short hand, 
\[
 \fP_B=\{X\in M(\C)^{\vg} : L_B(X)\succ 0\}. 
\]
 Spectrahedra and free spectrahedra occur in a number of areas
 of mathematics and its applications. For instance, spectrahedra are
 basic objects in semidefinite programming and convex optimization;
 free spectrahedra are connected to operator systems and spaces
 and complete positivity. Spectrahedra, both free and not, 
 arise in certain engineering applications. 
 For further details see the discussion in Subsection~\ref{s:prior}
  and \cite{bestmaps, BPT, convert-to-matin, emerge, WSV, SIG} and the references therein. 

 Free spectrahedra are \df{free sets} in the following sense: (i) if
 $X\in \fP_B[n]$ and $Y\in \fP_B[m],$ then 
\[
X\oplus Y = \left ( X_1\oplus Y_1,\dots, X_\vg \oplus Y_{\vg}\right) \in \fP[n+m],
\]
where
\[
 X_j\oplus Y_j = \begin{pmatrix} X_j & 0 \\ 0 & Y_j \end{pmatrix};
\]
and (ii) if $X\in \fP_B[n]$ and $U\in M_n(\CC)$ is  unitary, then
\[
 U^* X U = \left ( U^* X_1U, \dots, U^* X_\vg U \right)
\]
is in $\fP_B[n].$

\subsection{Free maps}
 A \df{free map} $f:\fP_B\to \fP_C$ between
 free spectrahedra is a sequence $f=(f[n]),$ where \index{$f[n]$}
 $f[n]:\fP_B[n]\to \fP_C[n],$ that satisfies: (i)
 if $X\in \fP_B[n]$ and $Y\in \fP_B[m],$ then
\[
 f[n+m](X\oplus Y) = f[n](X)\oplus f[m](Y); 
\]
 and (ii) 
 if $X\in \fP_B[n]$ and $U$ is an $n\times n$ unitary matrix, then
\[
 f[n](U^* XU)=  U^*f[n](X) U. 
\]
 It is customary to write $f$ in place of $f[n].$

  The free map 
 $f:\fP_B\to\fP_C$ is \df{analytic} if each $f[n]$ is analytic; and $f$ is
  \df{free bianalytic} %
  if $f$ is 
 analytic and has a free analytic inverse,
 $f^{-1}:\fP_C\to \fP_B.$  In the case $C=B,$
 a free bianalytic map is a  \df{free automorphism}.
  We will often just say
 $f$ is bianalytic or an automorphism, dropping the adjective free (and
 also analytic).  
  Motivations for studying free bianalytic
  maps between free spectrahedra
 stem from  connections with matrix inequalities that arise 
 in systems engineering \cite{emerge,convert-to-matin}  
 and by analogy with rigidity theory in several complex variables. 
 Free bianalytic maps can also be viewed as nonlinear
 completely positive maps. We expand upon these 
 themes 
 in Section~\ref{s:prior} below. 

 For a comprehensive treatment of free maps see \cite{KVV}.

\subsection{Hyper-Reinhardt free spectrahedra}
 We now introduce the central object of this paper. Fix
 positive  integers $d_1,\dots,d_{\vg+1}$  
 with $d=\sum d_j$ 
 and  norm one matrices $C_1,\dots,C_{\vg},$ where
 $C_j$ has size $d_j\times d_{j+1}.$  
 Let  $A_j$ denote the $(\vg+1)\times (\vg+1)$ block matrix
 with $(j,j+1)$ entry $C_j$ and all other entries $0.$ Thus
 each $A_j$ is of size $d\times d$ and the $(k,\ell)$
 block of $A_j$ has size $d_{k}\times d_\ell.$  Such a 
 tuple $A$ is now fixed for the remainder of this paper.
 The case $\vg=2$ was treated in the article \cite{MT}.

 Let \df{$\mathbb{T}$} denote the unit circle in the complex
 plane $\CC.$ A domain $D\subseteq \CC^\vg$ is 
 \df{circular} if $x\in D$ and 
 $\gamma\in\mathbb{T}$ implies $\gamma \, x\in D.$
  The domain $D$ is \df{Reinhardt} if it 
 satisfies the stronger condition: if $x\in D$
 and $\gamma=(x_1,\dots,x_\vg)\in \TT^\vg,$ then
\[
 (\gamma_1 x_1,\dots,\gamma_\vg x_\vg)\in D.
\]
 The Reinhardt domain $D$ is complete if $z\in D$
and $w\in \CC^\vg$ and $|w_j|\le |z_j|$
 for each $j,$ then $w\in D.$ Reinhardt domains
 are fundamental in the theory of functions of several
 complex variables, in part because 
 complete Reinhardt domains are the convergence
 domains of power series.  See \cite{first-steps}
 for more on Reinhardt domains.

 The free spectrahedron $\fP_A$  is Reinhardt in the sense that,  
given $\gamma\in \mathbb{T}^{\vg},$ 
 the map $\varphi_\gamma:M(\C)^\vg\to M(\C)^\vg$ given by
\[
 \varphi_\gamma(X)= \gamma\bcdot X
 :=(\gamma_1 X_1,\dots,\gamma_\vg X_\vg)
\]
 evidently restricts to an automorphism of $\fP_A.$
 We refer $\varphi_\gamma$ as a \df{trivial automorphism}
 and, because as explained by Lemma~\ref{l:U's},
 $\fP_A$ satisfies a stronger Reinhardt condition, we 
 call $\fP_A$ a \df{hyper-Reinhardt} free spectrahedron.

 The \df{free polydisc} in $\vg$-variables is the free set
 \[
  \mathscr{F}^{\vg} =\{X\in M(\C)^{\vg}: \|X_j\|<1\}.
\]
 It is easily seen to be a hyper-Reinhardt free spectrahedron.
 Because the $C_j$ have norm one, 
 $\fP_A\subseteq \mathscr{F}^{\vg}.$  The free polydisc
 $\mathscr{F}^\vg$  has
 a rich automorphism group \cite{Pop,MT16}. Properly interpreted,
 it is the same as that of the polydisc $\mathbb{D}^\vg,$
 where \df{$\mathbb{D}$} is the unit disc in the complex plane.

 Given $\mfI \subseteq \{1,2,\dots,\vg\},$ define \index{$\pi_{\mfI}$}
 $\pi_{\mfI}:M(\CC)^{\vg}\to M(\CC)^{\vg}$ by 
\[
 \pi_{\mfI}(T)_j= \begin{cases} T_j & \mbox{ if } j\notin \mfI \\
                              0 &\mbox{ if } j\in \mfI. \end{cases}
\]
 It is immediate that if $T\in \fP_A,$ then $\pi_{\mfI}(T)\in \fP_A.$
 As a definition,   $\fP_A$ \df{contains a polydisc as a
 distinguished summand} if
 there exists $\varnothing \ne \mfI \subseteq \{1,2,\dots,\vg\}$
 such that, for  $T\in M(\C)^{\vg},$ if
 $\|T_j\|<1$ for $j\in \mfI$ and $\pi_{\mfI}(T)\in \fP_A,$
 then $T\in \fP_A.$  If $\fP_A$ contains the polydisc in $\vh\le \vg$ variables
 as a distinguished summand, 
 then the automorphism group of $\fP_A$ contains a copy
 of the automorphism group of $\mathscr{F}^\vh$ and is thus
 rather large.

 Similarly, $\fP_A$ is a \df{coordinate direct sum} if
 there exists a $1\le \nu <\vg$ such that $T\in M(\CC)^{\vg}$
 is an element of $\fP_A$ if and only 
 both $\pi_{\mfI}(T)\in \fP_A$ and $\pi_{\widetilde{\mfI}}(T)\in \fP_A,$
 where $\mfI=\{1,2,\dots,\nu\}$ and $\widetilde{\mfI}$ is the
 complement of $\mfI$ in $\{1,2,\dots,\vg\}.$
 The condition  $\fP_A$ does not contain
 a coordinate direct summand is a natural condition that
 rules out permutations of the variables as an automorphism.

 Theorem~\ref{t:main} below is the main result of this article.

\begin{theorem}\label{t:main}
 Suppose $\varphi$ is an
 automorphism of  $\fP_A.$ %
\begin{enumerate}[(i)]
\item \label{i:main1} If $\fP_A$ 
  does not  contain a polydisc  as a distinguished summand,
 then $\varphi(0)=0;$
\item \label{i:main2} If $\varphi(0)=0,$ then there exists
 a $\gamma\in \mathbb{T}^{\vg}$ and a permutation $\pi$
 of $\{1,2,\dots,\vg\}$ such that 
 $\varphi_j(x)=\gamma_jx_{\pi(j)};$
\item \label{i:main3} If $\varphi(0)=0$ and $\fP_A$ is not 
 a coordinate direct sum,  then  $\varphi$ is trivial.
\end{enumerate}
\end{theorem}

  The proof of Theorem~\ref{t:main} 
 occupies the body of this paper.

\begin{corollary}
 \label{c:main}
 If $\fP_A$ is neither a polydisc nor a direct sum of free spectrahedra,
 then the automorphisms of $\fP_A$  are trivial.
\end{corollary} 

 The case of $\vg=2$ of Corollary~\ref{c:main} is one of the 
 principle results in \cite{MT}, though the proof here
 diverges materially from that in \cite{MT}.

\begin{proof}
 If $\fP_A$ is not a direct sum, then it does not contain a polydisc
as a proper distinguished summand.  By assumption $\fP_A$  is not a polydisc.
Hence $\fP_A$ does not contain a polydisc as a distinguished summand
and therefore, by Theorem~\ref{t:main} item~\eqref{i:main1} its automorphisms 
 send $0$ to $0.$  An application of Theorem~\ref{t:main}  item~\eqref{i:main3} now says
 the automorphisms of $\fP_A$ are trivial, since, by assumption, 
 $\fP_A$ does not a coordinate direct sum. 
\end{proof}

\subsection{Context and prior results}
\label{s:prior}
 This section expands upon context and prior results mentioned
 earlier in this introduction.

\subsubsection{Spectraballs, circular symmetry and prior results}
 Given a tuple $E=(E_1,\dots,E_{\vg})$ of $k\times \ell$ 
 (not necessarily square) matrices, let
\[
 B_j =\begin{pmatrix} 0 & E_j \\0& 0\end{pmatrix}.
\]
 The spectrahedron 
 $\fP_B,$ where $B\in M_{k+\ell}(\C)^{\vg}$ is, by definition,
 a \df{spectraball}. It is routine to see that a tuple
 $X\in M(\CC)^{\vg}$ is in $\fP_B$ 
 if and only if  $\|\Lambda_E(X)\|<1.$   
 We mention two special cases. One is
 the free polydisc  $\mathscr{F}^{\vg}$ 
\[
  \mathscr{F}^{\vg} =\{X\in M(\C)^{\vg}: \|X_j\|<1\}
\]
in $\vg$ 
variables introduced earlier. The other is the 
free $\vg\times \vh$ \df{matrix ball},
\[
 \mathscr{M}^{\vg\times \vh} =\{X =(X_{j,k})_{j,k=1}^{\vg,\vh}: X_{j,k}\in M(\C).
 \ \ \|X\|<1\}.
\]
 In each case, the automorphisms at the scalar level - those 
 of $\mathscr{F}^{\vg}[1]$ and $\mathscr{M}^{\vg\times \vh}[1]$ - naturally
 extend to free automorphisms and every free automorphism 
 arises this way.  See for instance \cite{Pop,MT16}.
 In particular,  automorphisms of $\mathscr{F}^{\vg}$ are composites 
 of permutation of the coordinates with M\"obius maps of the unit disc
 in each coordinate.

 More generally, the results of \cite{longmaps,bestmaps} characterize
 bianalytic maps between free spectrahedra $\fP_B$ and $\fP_C$
 under certain generic irreducibility hypotheses on 
 the tuples $B$ and $C.$  They also characterize bianalytic
 maps between spectraballs absent any additional hypotheses.
 A canonical class of free spectrahedra  not covered  by these
 results are those with circular symmetry.

 From \cite{billvolume}, if two circular free spectrahedra $\fP_A$
 and $\fP_B$ are bianalytic, then they 
 are linearly equivalent. Thus, in this case, there is an automorphism
 that sends $0$ to $0.$ In the other direction, and as a version
 of the corresponding classical Caratheodory-Cartan-Kaup-Wu (CCKW)
 Theorem from  several complex variables,
 an automorphism of a circular free spectrahedron
 that sends $0$ to $0$ is a (free) linear map \cite{proper}.
 As a conjecture, an automorphism of 
 a circular free spectrahedron that does not contain 
 a spectraball as a direct summand
 must send $0$ to $0$ and is therefore linear.

\subsubsection{Complete positivity}
 Given $C\in M_d(\CC)^\vg,$ let \index{$\cS_C$}
\[
 \cS_C =\operatorname{span} 
  \{I,C_1,\dots,C_\vg,C_1^*,\dots, C_\vg^*\} \subseteq M_d(\CC).
\]
  In particular,
 $\cS_C$ is an \df{operator system}.
 Given free spectrahedra $\fP_A$ and $\fP_B$ (both in $\vg$
 variables),  a $\vg\times \vg$ matrix $M$ (linear map) induces
 a completely positive map $M:\fP_A\to\fP_B$ if and only if
 $X\in \fP_A$ implies $MX=(Y_1,\dots,Y_\vg)\in \fP_B,$ where
\[
 Y_j=\sum_k M_{j,k}X_k,
\]
 as shown, for instance, in \cite{circular}. Thus $M$
 is a bijection from $\fP_A$ to $\fP_B$ if and only if
 $M$ is completely isometric.   By analogy, free bianalytic maps
 between free spectrahedra can be viewed as non-linear complete
 isometries.

\subsubsection{Automorphism groups and change of variables}
 Of course understanding the automorphism group of a free spectrahedron
 has its own intrinsic interest.   Moreover,
 as described in \cite{annals},  free semialgebraic sets,
 those described by (matrix-valued) polynomial inequalities 
 in the sense of positive semidefiniteness,   appear
 canonically in certain systems engineering problems.
 A related problem is to map, if possible, a general free semialgebraic  set
 $S$ bianalytically to a free specrahedron $\fP_A.$   Another
 free bianaltyic map from $S$ to say $\fP_B$ induces an
 automorphism of $\fP_A.$  Thus understanding the automorphism
 group of $\fP_A$ sheds some light on those free semialgebraic
 sets that are binaltyically equivalent to  the convex set $\fP_A.$

\section{Carath\'eodory Interpolation Preliminaries}
\label{s:car}
 In this section consequence of, and results related to, Carath\'eodory 
 interpolation needed in the sequel are collected for easy reference.
 We begin with the following well known  lemma.

\begin{lemma}
 \label{l:Cintstart}
 Given  $c_0,c_1\in \CC,$ let 
\[
 X=\begin{pmatrix} c_0 & c_1 \\0 &c_0\end{pmatrix}.
\]
 \begin{enumerate}[(i)]
  \item \label{i:c1} $\|X\|=1$ if and only 
  $0=1-|c_0|^2-|c_1|,$ and hence 
 there is a $\theta\in \RR$
 such that $c_1 = e^{i\theta} (|c_0|^2-1);$
 \item \label{i:c2} $\|X\|< 1$ if and only if 
  $0< 1-|c_0|^2-|c_1|.$
\end{enumerate}
 Further, in the first  case, the kernel of $I-XX^*$ is spanned by the vector
\[
 v=\begin{pmatrix} 1\\ -e^{-i\theta} c_0 \end{pmatrix}.
\]
\end{lemma}

\begin{proof}
 The matrix $X$ has norm $1$  if and only if $I-XX^*$
  is positive semidefinite with a kernel. In particular, either
 $|c_0|=1$ and $c_1=0;$ or $|c_0|<1$ and the determinant,
 namely  $1-|c_0|^2-|c_1|,$  of
 $I-XX^*$ is $0$ proving item~\eqref{i:c1}.  Similarly,
 $X$ has norm strictly less than $1$ if and only if  
  $|c_0|<1$ and the determinant of $I-XX^*$ is strictly positive.
   Direct computation
 verifies that $v$ is in the kernel of $I-XX^*$ in the case 
 that $X$ has norm $1.$
\end{proof}

 One view of Carath\'eodory interpolation is as follows.
 Let $\cS$ denote the shift operator on $H^2(\mathbb{D}),$
  the Hilbert-Hardy space
 of the unit disc.
 Given a positive integer $n$ and $c_0,c_1,\dots,c_n\in\CC,$
 does there exist a sequence $c_{n+1},c_{n+2},\dots\in\CC$ such that the 
 operator
\[
 \sum_{j=0}^\infty c_j \cS^j
\]
 has norm at most one. It turns out the answer is yes 
 if and only if the evident necessary condition that, letting
 $\cT^*$ denote the restriction of $\cS^*$ to the span
 of $\{1,z,\dots,z^n\},$ the matrix
\[
 g(\cT)=\sum_{j=0}^n c_j \cT^j
\]
 has norm at most one. In particular,  one can
 construct the $c_{n+j}$ {\it one step at a time}.
 Moreover, in the extreme case when $g(\cT)$
 has norm $1$  there is only one solution of the interpolation 
 problem. Lemma~\ref{l:Cintstart} above is connected with the case $n=1.$
 In Lemma~\ref{l:evenmoreCint} below
 the shift $\cS$ is replaced by a weighted shift.

\begin{lemma}
 \label{l:evenmoreCint}
 Suppose 
\begin{enumerate}[(i)]
 \item  $n\ge 2;$
 \item $c_0,\dots,c_n\in \CC$ and $|c_0|<1;$
 \item $\theta\in \RR$ and $c_1 =  e^{i\theta} (|c_0|^2-1).$
\end{enumerate}

 Let $\lambda_1=1,$  fix $\lambda_2,\dots,\lambda_n\in\CC\setminus\{0\}$
  and let $S=(S_{j,k})_{j,k=1}^{n+1}$ denote the $(n+1)\times (n+1)$ matrix 
  with $S_{j,j+1}=\lambda_j$ for $1\le j\le n$
  and $S_{j,k}=0$ otherwise. Thus the nonzero
 entries of $S$ are exactly on the first super diagonal. 
  Let
\[
 T= \sum_{j=0}^n c_j S^j \, \in \, M_{n+1}(\CC).
\]
 The following conditions are equivalent.
\begin{enumerate}[(i)]
\item the matrix $T$ has norm at most $1;$
\item  $c_j=e^{i\theta} (e^{i\theta}c_0^*)^{j-1} (|c_0|^2-1)$
 and $|\lambda_j|\le 1$ for  $j\ge 1;$ 
 \item $T$ has norm $1.$
\end{enumerate}

 Moreover, in this case, letting $P$ denote the $(n+1)\times (n+1)$
 matrix with $P_{1,n+1}=1$ and $P_{j,k}=0$ otherwise, if 
 $\mu\in \CC$ and $T+\mu P$ is a contraction (has norm at most $1$), then $\mu=0.$
\end{lemma}

\begin{proof}
 First note, for each $1\le j \le n,$ that
\begin{equation*}
X_j:= \begin{pmatrix} T_{j,j} & T_{j,j+1}\\ T_{j+1,j} & T_{j+1,j+1} \end{pmatrix}
  = \begin{pmatrix} c_0 & \lambda_j c_1\\0&c_0 \end{pmatrix}.
\end{equation*}
 Observe that $X_1$ 
 is  the matrix $X$ from Lemma~\ref{l:Cintstart} and in 
 particular has norm $1.$  Hence  $T$ has norm
 at least one and is thus a contraction if and only if it
 has norm one. Further, if $T$ is a contraction, then
 each $X_j$ is a contraction and hence, by Lemma~\ref{l:Cintstart},
$|\lambda_j|\le 1.$
 
 Suppose now the norm of $T$ is one and write $T$ as
\[
 T= \begin{pmatrix} X& Y\\ 0 & Z \end{pmatrix},
\]
 where $X=X_1.$  Using Lemma~\ref{l:Cintstart} and its notation,
  since $T$ is a contraction and $\|X^*v\|=\|v\|,$
\[
\|v\|^2 \ge  \| T^* \begin{pmatrix} v\\ 0 \end{pmatrix} \|^2
 = \| \begin{pmatrix} X^*v \\  Y^*v \end{pmatrix}\|^2
 =\|v\|^2 +\|Y^*v\|^2.
\]
 Thus the columns of $Y$ are orthogonal to $v$ and are thus multiples
 of 
\[
 w=\begin{pmatrix} e^{i\theta} c_0^* \\ 1 \end{pmatrix}.
\]
  On the other hand, 
 $y_k,$ the $k$-th column of $Y$ is given by
\[
 y_k = \lambda_1 \cdots \lambda_{k}
  \begin{pmatrix} c_{k} \\ c_{k-1} \end{pmatrix},
\]
 for $1\le k\le n.$  
 Using the assumption that the $\lambda_j$ are non-zero 
 and $y_{k+1}=\mu_{k+1} w$ for some $\mu_{k+1}\in \CC$ it follows
\[
  \begin{pmatrix} c_{k+1} \\ c_{k} \end{pmatrix} 
   = c_{k} \begin{pmatrix} e^{i\theta} c_0^* \\ 1 \end{pmatrix}.
\]
 Thus,  $c_{k+1}=e^{i\theta} c_0^* c_{k}.$
 In particular,  $c_2 = e^{2i\theta}c_0^*(|c_0|^2-1)$
  and by induction,
 $c_k=e^{ik\theta}(c_0^*)^{k-1} (|c_0|^2-1).$ 

 To prove the converse,  suppose $|\lambda_j|\le 1$ and
 $c_j= e^{ij\theta}c_0^{*(j-1)}(|c_0|^2-1).$ 
 In particular, $S$ is a contraction and 
\[
 f(z)=\frac{c_0-e^{i\theta}z}{1-c_0^* e^{i\theta}z}
 = \sum_{j=0}^\infty c_j z^j
\]
  is an automorphism of $\mathbb{D}.$ Hence $T=f(S)$
 is a contraction.

 Finally suppose $T=f(S)$ and $T+ \mu P$ is a contraction.
 The argument above, with $T+\mu P$ in place of $T,$  shows $y_{n} + \mu e$ 
 is orthogonal to the vector  $v$ from Lemma~\ref{l:Cintstart},
 where $e\in\CC^2$ is the vector with first entry $1$ and
 second entry $0.$   
 Since $y_{n}$ is orthogonal to $v,$ it follows that $\mu e$
 is orthogonal to $v.$ Hence
 $\mu=0$ and the proof is complete. 
\end{proof}

\section{The Affine Linear Terms}
Suppose $\varphi:\fP_A\to\fP_A$ is bianalytic,
\[
 \varphi=\begin{pmatrix} \varphi_1 & \cdots & \varphi_{\vg}
  \end{pmatrix}.
\]
 Let $b_j =\varphi_j(0).$ \index{$b_j$}

The coordinate functions $\varphi_j$ of $\varphi$
have power series expansions,
\[
 \varphi_j = b_j + \sum_{k=1}^{\vg} \ell_{j,k} x_k + h_j(x),
\]
 that converge in some neighborhood of $0$ and for 
 all nilpotent tuples, where \df{$h_j(x)$} consists of
 terms of order two and higher.  For notational purposes, let
\begin{equation}
\label{d:derL}
 \derL=\begin{pmatrix} \ell_{j,k} \end{pmatrix}_{j,k=1}^{\vg}.
\end{equation}
 Since $\varphi$ is bianalytic, $\derL$ is invertible. 
 \index{$\derL$}

 For integers $N\ge 2,$ let \df{$\sW_N$} denote the words  in $\{x_1,\dots,x_{\vg}\}$
 of length between $2$ and $N$. 
 Let \df{$\sW^j_N$} denote the words in 
 $\sW_N$ not of the form
 $x_j^n$ (for $2\le n\le N$). Let \df{$\sW^j$} denote the union
(over $N\ge 2$)  of the $\sW_N^j$ and let
 \df{$\fH^j$} denote power series of the form
\begin{equation*}
 \fh(x) = \sum_{\alpha \in \sW^j} a_{\alpha} x^\alpha.
\end{equation*}
  Let $\{\delta_1,\dots,\delta_{\vg}\}$ denote the standard
 orthonormal basis for $\CC^\vg.$ \index{$\delta_k$}
 Given $p\in \mathbb{D},$ let \index{$\fm$}
\begin{equation*}
 \fm_p(z) = \frac{p+z}{1+p^*z}.
\end{equation*}
 Proposition~\ref{l:ismob}  below 
 is the main results of this section.

\begin{proposition} \label{l:ismob}
 There exists  a permutation $\pi$ of $\{1,2,\dots,\vg\}$
 and a tuple $\theta=(\theta_1,\dots,\theta_\vg)\in \RR^\vg$
 such that, for each $1\le j\le \gv,$
\[
 \varphi_j(\delta_{\pi(j)} z)=  \fm_{b_j}(e^{i\theta_j}  z)
 = \frac{b_j+e^{i\theta_j}z}{1+b_j^* e^{i\theta_j}z}
   = e^{i\theta_j} \fm_{e^{-i\theta_j} b_j}(z).
\]
  More generally, with $k=\pi(j),$ there exists an $\fh_k\in \fH^k$ such that
\[
\varphi_j(x) = \fm_{b_j}(e^{i\theta_j}x_{k})
  +\fh_k(x).
\]
 In particular, there exists $a^k_\alpha\in \CC$ such that
  if $T$ is nilpotent of order $N+1,$ then
\begin{equation}
\label{e:vphj}
 \varphi_j(T) = \fm_{b_j}(e^{i\theta_j}T_{k})
 + \fh_k(T)  =  \fm_{b_j}(e^{i\theta_j}  T_{k})
   + \sum_{\alpha\in\sW^k_N} a^k_\alpha T^\alpha.
\end{equation}
\end{proposition}

 We refer to $\pi$ as the \df{permutation associated to}
 $\varphi.$

 Before proving Proposition~\ref{l:ismob}, we pause 
 to prove item~\eqref{i:main2}  of Theorem~\ref{t:main},
 stated as Proposition~\ref{p:main2} below.

\begin{proposition}
 \label{p:main2}
  If $\varphi(0)=0,$ then there exists $\gamma\in\mathbb{T}^{\vg}$
 and permutation $\pi$ such that $\varphi_j(x)=\gamma_j x_{\pi(j)}.$
\end{proposition}

\begin{proof} 
 Let $\pi$ denote the permutation from Proposition~\ref{p:main2}.
 Since  $\varphi$ is an
 automorphism of a circular domain that maps $0$ to $0$
 it is linear by \cite[Theorem~4.4]{proper}.
 On the other hand, by Proposition~\ref{l:ismob},
 $\varphi_j(x) = e^{i\theta_j}x_k + \fh_k(x),$ 
 where $k=\pi(j).$
 Linearity forces $\fh_k=0$ completing the proof.  
\end{proof}

The proof of Proposition~\ref{l:ismob} is broken down into
several lemmas.

\begin{lemma}
 \label{l:eees}
 There is a permutation $\pi$ of $\{1,\dots,\vg\}$ 
 and a tuple $\theta\in \RR^{\vg}$ such that, 
   for tuples $T=(T_1,\dots,T_{\vg})$ nilpotent of order $2,$
\[
 \varphi_j(T) = b_j + e^{i\theta_j} (|b_j|^2-1)  T_{\pi(j)}
\]
\end{lemma}

Before proving Lemma~\ref{l:eees}, we first 
 establish Lemma~\ref{l:eeesbegin} below.
 Recall equation~\eqref{d:derL}.

\begin{lemma}
 \label{l:eeesbegin}
  There is a permutation $\pi$ of $\{1,\dots,\vg\}$
 such that $\ell_{j,k}\ne 0$ if and only if  $k=\pi(j).$
 In particular, 
   for tuples $T=(T_1,\dots,T_{\vg})$ nilpotent of order $2,$
\[
 \varphi_j(T) = b_j + \ell_{j,\pi(j)}  T_{\pi(j)}.
\]
  Thus 
 $\derL$ has exactly one nonzero entry in each row and column.
\end{lemma}

The proof of Lemma~\ref{l:eeesbegin} in turn uses the following
 elementary linear algebra lemma included here for completeness.

\begin{lemma}
 \label{l:elementary}
  Suppose $R$ and $\dB$ are self-adjoint $n\times n$ matrices and
 $\gamma \in \CC^n.$ If $R\pm \dB\succeq 0$ and $R\gamma =0,$
 then $\dB\gamma=0.$
\end{lemma}

\begin{proof}
 It is immediate that $\langle \dB\gamma,\gamma \rangle =0.$
 For $\lambda\in \CC$ and $\delta\in \CC^n,$ the inequality
\[
 0\le  \langle (R+\dB)(\gamma +\lambda\delta), \gamma+\lambda \delta\rangle
  = |\lambda|^2 \langle (R+\dB)\delta,\delta\rangle
  +  \lambda \langle \dB\gamma, \delta\rangle +
  \lambda^* \langle \dB\delta,\gamma \rangle
\]
implies $\langle \dB\gamma,\delta\rangle =0$ and hence $\dB\gamma=0.$
\end{proof}

For $t\in\CC^\vg,$ define $\qT=(\qT_1,\dots,\qT_{\vg})=\qT(t)=\qT(t_1,\dots,t_{\vg})$ by
$\qT_j = t_j S,$ where  \index{$\qT(t)$}
\begin{equation}
 \label{d:firstS}
 S=\begin{pmatrix} 0 & 1 \\ 0 & 0 \end{pmatrix}.
\end{equation}

 Recall the definition of $\derL$ from equation~\eqref{d:derL}
 and let $B= I+ \Lambda_A(b) + \Lambda_A(b)^*.$  
 Since  since $\varphi(\qT(t))$ is affine linear in $t,$
\begin{equation}
 \label{e:afflin}
 \varphi(\qT(t))= b+ \qT(\derL t), 
\end{equation}
 it follows that 
\[
 L_A(\varphi(\qT(t))) = B  + \Lambda_A(\qT(\derL t)) +  \Lambda_A(\qT(\derL t))^*,
\]
 and  if
 $s,t\in \CC^{\vg},$ then
\[
 L_A(\varphi(\qT(t+s))) = L_A(\varphi(\qT(t))) + \Lambda_A(\qT(\derL s)) +  \Lambda_A(\qT(\derL s))^*.
\]

\begin{lemma}
 \label{l:RandB}
    If $t\in \CC^{\vg}$ and $|t_j|\le 1$
  for all $j,$ then $\qT(t)$ is in the closure of 
 $\fP_A.$ If in addition, there is an $1\le m\le \vg$  such that
 $|t_m|=1,$ then $\qT(t)$ is in the boundary of $\fP_A.$

 Given $1\le m\le \vg,$  if $s\in \CC^{\vg}$ satisfies
 $s_m=0$ and $|s_j|\le 1$ for all $j,$ then 
 $\varphi(\qT(\delta_m))$ and $\varphi(\qT(\delta_m+s))$ 
 are in the boundary of $\fP_A.$  Equivalently, 
\begin{enumerate}[(a)]
 \item  $R:=L_A(\varphi(\qT(\delta_m)));$ and
 \item $L_A(\varphi(\qT(\delta_m\pm s)))$
\end{enumerate}
 are positive semidefinite with non-trivial kernel.
 In particular,  
 there is a vector $0\ne \gamma$ such that $R\gamma =0.$

 Moreover,
\begin{equation}\label{e:portion}
 L_A(\varphi(\qT(\delta_m\pm s))) = R \pm \dB(s),
\end{equation}
 where
\[
 \dB(s) = \Lambda_A(\qT(\derL s)) + \Lambda_A(\qT(\derL s))^*.
\]

Finally, $\dB(s)\gamma=0.$
\end{lemma}

\begin{proof}
 The first statements about $\qT(t)$ are evident.

 By the first part of the lemma,
  $\qT(\delta_m)$ and $\qT(\delta_m\pm s)$ 
 are in the boundary of $\fP_A$ and hence 
 so are there images under $\varphi.$

 Proving equation~\eqref{e:portion} is straightforward
 based upon equation~\eqref{e:afflin}.
 The
 last statement follows from equation~\eqref{e:portion}.
 and Lemma~\ref{l:elementary}.
\end{proof}

\begin{proof}[Proof of Lemma~\ref{l:eeesbegin}]
 Recall, $\{\delta_1,\dots,\delta_{\vg}\}$ is the standard
 orthonormal basis for $\CC^{\vg}.$
  Let $H_j$ denote the orthogonal complement of $\delta_j$
 and suppose, for each $1\le k\le \vg,$ there is a 
 $1\le \pi(k) \le \vg$ such that $\derL H_k \subseteq H_{\pi(k)}.$
 Since $\derL$ is invertible, it follows that $\derL H_k = H_{\pi(k)}$
and $\pi(k)\ne \pi(j)$ for $j\ne k.$ In particular,
 $\pi$ is a permutation. Moreover,  for each
 $1\le j\le \vg,$ using $\derL$ is invertible gives
\[
 \derL [\delta_j]= \derL \cap_{k\ne j} H_k = \cap_{k\ne j} \derL H_k
   = \cap_{k\ne j} H_{\pi(k)} = [\delta_{\pi(j)}],
\]
 where $[x]$ denotes the span of a vector $x\in \CC^{\vg}.$
 Thus, for $j$ fixed,  $\ell_{j,k}\ne 0$ if and only if $k=\pi(j)$
  and the conclusion of the lemma  holds. 

 Now suppose there is an $1\le m\le \vg$ such that
 $\derL H_m \not\subseteq  H_k$ for each $k.$  It follows that 
 $\dim [\derL H_m] \cap H_k<\dim H_k=\vg-1$ and consequently
 $\cup_{k=1}^\vg [\derL H_m] \cap H_k \subsetneq \derL H_m.$
 Thus there exists an $s\in H_m\subseteq \CC^{\vg}$ such that each entry 
 of $\derL s$ is non-zero. By scaling $s,$ we assume
 the entries of $s$ have absolute value at most one. 
 Applying Lemma~\ref{l:RandB} (and using its notation),
  fix a (non-zero) vector $\gamma$ in the kernel of $R$
 so that $\dB(s)\gamma=0.$

 Expressing
 $\gamma = \oplus_{j=1}^{\vg+1} \gamma_j$ and
\begin{equation}
\label{d:gj}
 \gamma_j = \begin{pmatrix} \gamma_{j,1}  \\  \gamma_{j,2}\end{pmatrix}
  \in \CC^{d_j}  \oplus  \CC^{d_j},
\end{equation}
 and setting $\sigma =\derL s,$  the condition $\dB(s)\gamma=0$ implies
 $0  = \sigma_j \,   C_j \gamma_{j+1,2}$ and 
 $0  =  \sigma_j^* \,  C_j^*  \gamma_{j,1}$ for 
 $1\le j\le \vg.$ 
 Thus, 
\begin{equation}
\label{e:somezeros}
 C_j^*\gamma_{j,1}=0=C_j\gamma_{j+1,2}, 
\end{equation}
 since $\sigma_j\ne 0.$  In particular,
 $C_{\vg}^* \gamma_{\vg,1}=0.$ 

Since $R\gamma=0,$ 
\[
 b_{\vg}^* C_\vg^* \gamma_{\vg,1} + \gamma_{\vg+1,1} = 0
\]
 and it follows that $\gamma_{\vg+1,1}=0.$ 

 Arguing by induction,
 suppose $\gamma_{\vg+1,1}, \dots, \gamma_{k+1,1}$ are
 all $0$ for some $2\le k\le \vg.$  The equality $R\gamma =0$
 implies
\begin{equation}
 \label{e:morezeros}
 0=  b_{k-1}^* C_{k-1}^* \gamma_{k-1,1} 
  + \gamma_{k,1} + b_k C_k\gamma_{k+1,1}
  + \ell_{k,m} C_k \gamma_{k+1,2}.
\end{equation}
 It follows from equations~\eqref{e:somezeros} 
 and \eqref{e:morezeros} and the 
 induction hypothesis
  that $\gamma_{k,1}=0.$  Hence $\gamma_{k,1}=0$
 for all $2\le  k\le \vg+1.$ Finally, 
 from $\gamma_{1,1}+b_1 C_1\gamma_{2,1} =0$
 it follows that $\gamma_{1,1}=0.$
 A similar argument starting from
\[
 0= \gamma_{1,2} + b_1 C_1\gamma_{2,2} 
\]
 shows $\gamma_{k,2}=0$ for each $1\le k\le \vg+1$ too.
 We have now reached the contradiction $\gamma=0,$ 
 completing the proof.
\end{proof}

\begin{proof}[Proof of Lemma~\ref{l:eees}]
 Let $\pi$ denote the permutation of Lemma~\ref{l:eeesbegin}. Thus,
  for each $1\le j\le \vg$ and
 each tuple $T$ nilpotent of order two,
\[
 \varphi_j(T) = b_j + \ell_{j,\pi(j)} T_{\pi(j)}.
\]

 Fix $1\le m\le \vg$ and let $v=\pi^{-1}(m).$  In particular,
  $\pi(v)=m.$ 
 Let $s\in \CC^{\vg}$ denote the tuple with $s_a=1$
 for $a\ne m$ and $s_m=0.$ Using Lemma~\ref{l:RandB}
 and its notation, there is a non-zero vector $\gamma$
 such that $R\gamma=0$ and $\dB(s)\gamma=0.$ In this case,
\[
 \dB(s) = \sum_{j\ne v} \ell_{j,\pi(j)} A_j \otimes S 
   + \left [ \sum_{j\ne v} \ell_{j,\pi(j)} A_j \otimes S\right ]^*
\]
 where $S$ is given in equation~\eqref{d:firstS}.

 With $\gamma = \oplus_{j=1}^{\vg+1} \gamma_j$ and
 $\gamma_j$ as in equation~\eqref{d:gj} like before,
 the condition $\dB(s)\gamma=0$ 
   (and $\ell_{j,\pi(j)}\ne 0$) implies, 
 for $j\ne v,$  
\begin{equation}
 \label{e:somezeros2}
 C_j^* \gamma_{j,1} =0 =   C_j \gamma_{j+1,2}. 
\end{equation}
 As in the proof of Lemma~\ref{l:eeesbegin}, $R\gamma =0$
 implies $\gamma_{\vg+1,1}+ b_\vg^* C_{\vg}^*\gamma_{g,1}=0$
 and also $\gamma_{1,2}+b_1 C_1 \gamma_{2,2}=0.$
 Assuming $v<\vg,$ it follows that $\gamma_{\vg+1,1}=0.$
 Another application of $R\gamma =0$ gives
\[
 b_{\vg-1}^* C_{\vg-1}^* \gamma_{\vg-1,1}
 + \gamma_{\vg,1} + b_{\vg} C_{\vg} \gamma_{\vg+1,1}
  + \ell_{\vg,\pi(\vg)} C_{\vg} \gamma_{\vg+1,2}  =0.
\]
  Thus, if $v<\vg-1,$ then, 
 in view of equation~\eqref{e:somezeros2} 
 and using the already established $\gamma_{\vg+1,1}=0,$
 it follows that $\gamma_{\vg,1}=0.$ By induction,
 $\gamma_{j,1}=0$ for $v+2\le j.$
A similar argument shows $\gamma_{j,2}=0$ for $j\le v-1.$
On the other hand, $R\gamma=0$ gives
\begin{equation}
\label{e:somezeros3}
 \begin{split}
 0= &    b_{v-1}^* C_{v-1}^* \gamma_{v-1,1} + \gamma_{v,1} +
  b_v C_v \gamma_{v+1,1} + \ell_{v,m} C_{v+1}\gamma_{v+1,2}
\\ 0 = &  \ell_{v-1,m} C_{v-1}^*\gamma_{v-1,1} +
 b_{v-1}^* C^*_{v-1}\gamma_{v-1,2} + \gamma_{v,2} 
 + b_v C_v\gamma_{v+1,2} 
\\ 0= & b_v^* C_v^* \gamma_{v,1} + \gamma_{v+1,1}
  + b_{v+1} C_{v+1} \gamma_{v+2,1}
  + \ell_{v+1,m} C_{v+1} \gamma_{v+2,2}
\\ 0 = &   \ell_{v,m}^* C_v^* \gamma_{v,1} + b_v^* C_v^* \gamma_{v,2}
 + \gamma_{v+1,2} + b_{v+1} C_{v+1}\gamma_{v+2,2}.
\end{split}
\end{equation}
 Combining equations~\eqref{e:somezeros2}, \eqref{e:somezeros3} and $\gamma_{v-1,2}=0=\gamma_{v+2,1},$
  it follows that
\[
 \begin{split}
 0= & \,   \gamma_{v,1} +
   b_v C_v \gamma_{v+1,1} + \ell_{v,m} C_{v}\gamma_{v+1,2}\\
0 = & \,   \gamma_{v,2} + b_v C_v\gamma_{v+1,2} \\
0= & \,  b_v^* C_v^* \gamma_{v,1} + \gamma_{v+1,1} \\
0 = & \,  \ell_{v,m}^* C_v^* \gamma_{v,1} + b_v^* C_v^* \gamma_{v,2}
 + \gamma_{v+1,2}.
\end{split}
\]
 Thus,
\[
0= \begin{pmatrix} I & \varphi_v(T(\delta_m)) \otimes C_v \\ 
   \varphi_v(T(\delta_m))^*\otimes C_v^* & I
  \end{pmatrix}\, \begin{pmatrix} \gamma_{v,1} \\ \gamma_{v,2}  \\
  \gamma_{v+1,1} \\ \gamma_{v+1,2} \end{pmatrix}
\]
 and hence 
\[
 \varphi_v(T(\delta_m)) = \begin{pmatrix} b_v & \ell_{v,m} \\ 0 & b_v \end{pmatrix}
\]
 has norm one or $\gamma_{a,b}=0$ for $a=v,v+1$ and $b=1,2.$ 
 In the second case, $L_A(c I_2)\gamma=0,$
 where $c\in \CC^{\vg}$ is the tuple with $c_j=b_j$ for $j\ne v$
 and $c_v=0.$ Thus $c$ is in the boundary of $\fP_A[1].$
 On the other hand, since $b$ is not in this boundary of $\fP_A[1],$
 it is
 easy to see that neither is $c,$ a contradiction. 
 In the first case, from Lemma~\ref{l:Cintstart},
 $\ell_{v,m} = e^{i\theta_v} (|b_v|^2-1),$ for some
 $\theta_v\in \RR$ and the proof is complete.
\end{proof}

\begin{proof}[Proof of Proposition~\ref{l:ismob}]
 By Lemma~\ref{l:eees}, there
  is a permutation $\pi$ of $\{1,\dots,\vg\}$ such that, 
   for tuples $T=(T_1,\dots,T_{\vg})$ nilpotent of order $2$
 and $1\le j\le\vg,$
\[
 \varphi_j(T) = b_j + e^{i\theta_j} (|b_j|^2-1)  T_{\pi(j)},
\]
 for some $\theta_j\in\RR.$
  Hence, with $k=\pi(j),$ for words $\alpha$ of length
 at least two,
  there exists $a^k_\alpha\in\CC$  such that for a tuple $T$ nilpotent of order 
 $N+1\ge 2,$
\[
 \varphi_j(T) = b_j+ e^{i\theta_j} (|b_j|^2-1) T_{\pi(j)} 
  + \sum_{\alpha\in\sW_N} a^k_\alpha T^\alpha.
\]

 For $N\in\mathbb N^+,$ let $R$ denote the $(N+1)\times (N+1)$
 matrix with $R_{j,j+1}=1$ for $1\le j\le N$ and $R_{j,k}=0$
 otherwise. 
 Let $k=\pi(j).$  The tuple $Y=\delta_k R$ 
 is nilpotent of order $N+1$ and is
 in the boundary of $\fP_A$ and hence so is
 $\varphi(Y).$ In particular, %
\[
 \varphi_j(Y) = b_j + e^{i\theta_j} (|b_j|^2-1) R
   + \sum_{n=2}^N a^k_{x_k^n} R^n.
\]
is a contraction. By Lemma~\ref{l:evenmoreCint}, 
$a^k_{x_k^n} =  e^{i n\theta_j}(b_j^*)^{n-1}(|b_j|^2-1).$
 Hence equation~\eqref{e:vphj} holds for $T$ nilpotent
 of order $N+1$ and the remainder of the proposition follows.
\end{proof}

\section{The Higher Order Terms}
\label{s:higher}
 We continue to work with a fixed automorphism $\varphi$
 of $\fP_A,$ with $b=\varphi(0).$ In particular,
 there exists a permutation $\pi$ and a tuple 
 $\theta\in \RR^{\vg}$ such that the conclusion
 of  Proposition~\ref{l:ismob} holds. Let \index{$\fg$}
\begin{equation}
 \label{d:gm}
   \fg_j(z) = e^{i\theta_j} \fm_{e^{-i\theta_j}b_j}(z) 
    =\fm_{b_j}(e^{i\theta_j}z) = b_j + e^{i\theta_j}(1-|b_j|^2)z  +
      \dots \, .
\end{equation}

 Let $\fZ^+ \subseteq \{1,2,\dots,\vg\}$ \index{$\fZ^+$}
 denote those indices $k$
 such that, for each each $\epsilon>0,$ the matrix
 $L_A(\delta_k+\epsilon \delta_{k+1})$ is not positive semidefinite;
 let $\fZ^-$ denote those indices $k$ \index{$\fZ^-$}
 such that, for each each $\epsilon>0,$ the matrix
 $L_A(\delta_k+\epsilon \delta_{k-1})$ is not positive semidefinite;
  and let $\fZ=\fZ^+\cup \fZ^-.$ \index{$\fZ$}
  Let \df{$\fN$} denote   the complement of $\fZ.$ It is 
  straight forward consequence of convexity that $\fN$ is the set of those
 indices $j$ for which there exists an $\epsilon>0$
 such that 
  $L_A(\epsilon \delta_{k-1} + \delta_k+\epsilon \delta_{k+1}) 
  \succeq 0,$ (in the case that $k=1$ or $k=\vg$ we interpret
  this condition in the natural way).

 The following alternate criteria for membership in the
 sets $\fZ^\pm$ will often be used without comment. 

\begin{lemma}
 \label{l:fZ+}
  For a tuple $T\in M_n(\CC)^\vg,$ if $L_A(T)\succeq 0,$ then
  for each $j,$
\[
  T_j^*T_j\otimes C_j^*C_j + 
     T_{j+1}T_{j+1}^*\otimes C_{j+1} C_{j+1}^* \succeq 0.
\]

   An index $j$ is in $\fZ^+$ if and only if, for each $\epsilon>0,$
\[
 C_j^* C_j + \epsilon C_{j+1}C_{j+1}^* \not\preceq I.
\]
 Likewise $j\in\fZ^-$ if and only if for each $\epsilon>0,$
 $\epsilon C_{j-1}^*C_{j-1} + C_jC_j^* \not \preceq I.$
\end{lemma}

The proof uses several times the Schur complement criteria
 for positive definiteness of a $2\times 2$ block selfadjoint matrix
\[
 M=\begin{pmatrix} A& B\\B^* &D \end{pmatrix}.
\]
Namely, if $A$  is positive definite, them $M$ is positive
 semidefinite if and only if 
  the Schur complement of $M$ (with respect
 to the $(1,1)$ block entry),
\[
 S= D-B^* A^{-1} B
\]
is positive semidefinite.

\begin{proof}
 If $L_A(T)\succeq 0,$ then 
\[
  Z=\begin{pmatrix} I & T_j\otimes C_j & 0 \\ T_j^*\otimes C_j^* &
  I & \epsilon  T_{j+1}\otimes C_{j+1}\\
  0 & \epsilon T_{j+1}^*\otimes  C_{j+1}^* & I \end{pmatrix}\succeq  0
\]
if and only if the Schur complement of $Z$ with respect to 
 the $(1,1)$ entry,
\[
  S=\begin{pmatrix} I-T_j^*T_j\otimes C_j^*C_j & T_{j+1}\otimes C_{j+1} \\  
    T_{j+1}^*\otimes  C_{j+1}^* & I \end{pmatrix}
\]
 is positive semidefinite
 if and only if the Schur complement of $S$ with respect
 to its $(2,2)$ entry,
\[
 I-T_j^*T_j\otimes C_j^*C_j -
  T_{j+1}T_{j+1}^*\otimes  C_{j+1}C_{j+1}^* 
\]
 is positive semidefinite. 
Equivalently, 
$T_j^*T_j\otimes C_j^*C_j +
  T_{j+1}T_{j+1}^*\otimes  C_{j+1}C_{j+1}^*\preceq I.$

 Observe  $L_A(\delta_j+\epsilon \delta_{j+1})\succeq 0$ 
 if and only if 
\[
  \begin{pmatrix} I & C_j & 0 \\ C_j^* & I & \epsilon C_{j+1}\\
  0 & \epsilon C_{j+1}^* & I \end{pmatrix}.
\]
 By the argument above it follow that 
 $L_A(\delta_j+\epsilon \delta_{j+1})\succeq 0$ if and only if
$C_j^*C_j+\epsilon^2 C_{j+1}C_{j+1}^*\preceq I.$
\end{proof}

 For notational convenience, let $k=\pi(j).$
  From Proposition~\ref{l:ismob}, there exists $\fh_k\in \fH^j$ such
 that
\[
 \varphi_j(x) = \fg_j(x_k) +\fh_k(x),
\]
 where $\fg_j$ is given in equation~\eqref{d:gm}.

\begin{lemma}
 \label{l:mobius+}
   Fix $1\le j\le \vg$ and let $k=\pi(j).$ 

\begin{enumerate}[(i)]
 \item \label{i:m+1}
  If $k\in \fN,$   then $\fh_k=0.$ 
\item \label{i:m+4}
 In any case, 
 $\fh_k(T)=0$ whenever $T_{k-1}=0=T_{k+1}.$
\end{enumerate}
\end{lemma}

Lemma~\ref{l:mobius+} is a preliminary version of 
Proposition~\ref{p:sumry}, the main result of this section.
 Before proving Lemma~\ref{l:mobius+},  
  we collect a couple of lemmas
 related to the hyper-Reinhardt nature of $\fP_A.$

\subsection{The hyper-Reinhardt condition}

Given $T\in M(\C)^{\vg},$ note 
\begin{equation*}
L_A(T) =  \begin{pmatrix} I & C_1\otimes T_1 & 0 & 0 & \dots &0 &0 \\
        C_1^*\otimes T_1^* & I & C_2\otimes T_2 & 0 & \dots &0 & 0 \\
        0 & C_2^*\otimes T_2^*  & I & C_3\otimes T_3 & \dots &0 & 0 \\
       0&0& C_3^*\otimes T_3^*&I&\dots &0&0\\
        \vdots &\vdots&\vdots& \vdots& \dots &  \vdots \\
       0&0&0&0 &\dots & I & C_{\vg}\otimes T_\vg \\
       0&0&0&0 &\dots &C_\vg^* \otimes T_\vg^* &  I 
 \end{pmatrix}. 
\end{equation*}

 Given tuples $T\in M_n(\CC)^{\vg}$ and 
 $W=(W_0,\dots,W_\vg)\in M_n(\CC)^{\vg+1},$ let \index{$W\AST T$}
\begin{equation}
 \label{d:AST}
 W\AST T := (W_0^*T_1 W_1, W_1^*T_2W_2,\dots, W_{\vg-1}^* T_{\vg} W_\vg).
\end{equation}

\begin{lemma}
 \label{l:U's}
   If $T\in \fP_A[n]$ and $W\in  M_n(\C)^{\vg+1}$ is a tuple of 
  unitary matrices, then $W\AST T$ is also in $\fP_A.$
\end{lemma}

\begin{remark}\rm
 A natural question is whether a free spectrahedra
 $\fP_B\subseteq M(\C)^{\vg}$ is hyper-Reinhardt if and only if 
 for each $n,$ each $X\in \fP_B[n]$ and tuple
 $W=(W_0,\dots,W_\vg) \in M_n(\CC)^{\vg+1}$ of unitary matrices
\[
 W\AST X = (W_0^*X_1W_1,\dots, W_{\vg-1}^* X_{\vg}  W_\vg)\in \fP_A. 
\]
\end{remark}

\begin{proof}[Proof of Lemma~\ref{l:U's}]
 Letting $D_W$ denote the block diagonal $(\vg+1)\times (\vg+1)$
with  $(j+1,j+1)$ entry $I\otimes W_j$ for
 $0\le j\le \vg,$ 
\[
 D_W^* L_A(T) D_W = L_A(W\AST T). \qedhere
\]
\end{proof} 

 \begin{lemma}
 \label{l:KR+}
  If $\lambda\in \fP_A[1],$ and  $X\in M_n(\C)^{\vg}$ is a tuple of 
 contractions, then 
 $\lambda \bcdot X =
   (\lambda_1X_1,\dots,\lambda_{\vg} X_{\vg})\in\fP_A[n].$
\end{lemma}

\begin{proof}
 Given a tuple $U\in M_n(\C)^{\vg}$ of unitary matrices,
 let $W_0=I$ and $W_j=U_1 U_2\dots U_j$ and let $\Lambda= \lambda I_n.$
 Observe that $\Lambda\AST W = \lambda\bcdot U.$
 Thus
\[
 L_A(\Lambda \AST W) = L_A(\lambda\bcdot U)
\]
 and consequently,  by Lemma~\ref{l:U's},
   $\lambda\bcdot U \in \fP_A.$
 
 By the Russo-Dye Theorem, there is an $n$ and 
 unitary tuples $U^{j}\in M_n(\C)^{\vg}$ for $1\le j\le n$ such that
\[
 X = \frac{1}{n} \sum_{\ell=1}^n U^\ell, \ \ 
 X_j=\frac1n \sum_{\ell=1}^n U_j^\ell.
\] 
 Thus
\[
 L_A(\lambda\bcdot X) = \frac{1}{n} \sum_{j=1}^n L_A(\lambda\bcdot U^j)
 \succ0
\]
 and it follows that $\lambda\bcdot X\in \fP_A.$
\end{proof}

We will also need the following variation of Lemma~\ref{l:KR+}.

\begin{lemma}
 \label{l:KR++}
   Fix $\cI \subseteq \{1,2,\dots,\vg\}$ and 
 $\lambda_j\ge 0$ for $j\in \cI.$
  If $Y\in \fP_A[n]$ and $Y_j=0$ for $j\in \cI$  implies $Y^\lambda$
 defined by 
\[
   Y^\lambda_j = \begin{cases} \lambda_j I_n & \mbox{ if } j\in \cI \\
                              Y_j & \mbox{ if } j\notin \cI,
 \end{cases}
\]
 is in $\fP_A,$ 
 then, if $Y\in \fP_A$ and $Y_j=0$ for $j\in \cI$ and $\|T_j\|\le 1$
 for $j\in \cI,$   then the tuple
\[
 Z = \begin{cases} \lambda_j T_j & \mbox{ if } j\in \cI \\
                              Y_j & \mbox{ if } j\notin \cI
 \end{cases}
\]
 is in $\fP_A.$
\end{lemma}

\begin{proof}
 Given a tuple $W=(W_0,W_1,\dots,W_\vg)$  of $n\times n$ unitary matrices,
 the tuple $W^*\AST Y\in \fP_A$  and $(W^*\AST Y)_j=0$
 for $j\in \cI.$  Hence the tuple $(W^*\AST Y)^\lambda$
 is in $\fP_A$ and therefore so is
\[
  W\AST (W^*\AST Y)^\lambda
   = \begin{cases}  \lambda_j W_j^* W_{j+1} & \mbox{ if } j\in \cI \\
                          Y_j & \mbox{ if } j\notin\cI.\end{cases}
\]
  An application of the Russo-Dye lemma as in the proof
 of Lemma~\ref{l:U's} completes the proof.
\end{proof}

\begin{lemma}
 \label{l:KRapplied}
   Fix $1\le k\le \vg.$ If there is an $\epsilon>0$ such
 that both $L_A(\epsilon \delta_{k-1} + \delta_k) \succeq 0$
 and  $L_A(\delta_k+\epsilon \delta_{k+1}) 
  \succeq 0,$ then there is an $\eta>0$ such that 
 if $T\in M_n(\C)^\vg$ satisfies $\|T_k\|\le 1$
 and $\|T_j\|\le \eta$ for $j\ne k,$ then $T$ is in 
  the closure  of $\fP_A[n].$ If, in addition, $\|T_k\|=1,$
 then $T$ is in the boundary of $\fP_A[n].$ %

 Similarly, if just $L_A(\delta_k +\epsilon \delta_{k+1})\succeq0,$
 then there is an $\eta>0$ such that if 
 $\|T_k\|\le 1,$ $T_{k-1}=0$ and $\|T_j\|<\eta$ for 
 $j\ne k,$ then $T$ is in the closure of $\fP_A[n]$.
\end{lemma}

\begin{proof}
 Observe that there is a $\sigma>0$ such that
 $L_A(\delta_k + \sigma \sum_{j\ne k,k+1,k-1} \delta_j) \succeq 0.$
 Equivalently $ \delta_k +\sigma \sum_{j\ne k,k+1,k-1} \delta_j$
 is in the closure of $\fP_A.$
 By convexity, the average of $\delta_k + \sigma \sum_{j\ne k,k+1,k-1} \delta_j$
 and $\delta_k + \epsilon \delta_{k+1}$ and $\delta_k + \epsilon \delta_{k-1}$
 is also in $\fP_A.$ Thus, there is an $\eta>0$ such that
$ L_A(\delta_k +\eta \sum_{j\ne k} \delta_j)  \succeq 0.$ 
 From here an application  of Lemma~\ref{l:KR+}, with
  $X_k=T_k$ and $X_j = \frac{T_j}{\eta}$ for $j\ne k$
 establishes the first part of the  lemma.

 To prove the second part of the lemma, note that there is an $\eta>0$
 such that 
 $L_A(\delta_k + \eta \sum_{k-1\ne j\ne k} \delta_j)\succeq 0$
 and argue as above.
\end{proof}

\subsection{Proof of Lemma~\ref{l:mobius+}}
  Let $k=\pi(j).$
  From Proposition~\ref{l:ismob}, there exists $\fh_k\in \fH^j$ such
 that
\[
 \varphi_j(x) = \fg_j(x_k) +\fh_k(x),
\]
 where $\fg_k$ is given in equation~\eqref{d:gm}.

 Now suppose $k\in \fN.$ 
 We argue by induction on $N$ that the result 
 of item~\eqref{i:m+1} of Lemma~\ref{l:mobius+} holds
 when evaluating $\varphi_j$ at a tuple $T$ that
 is nilpotent of order $N.$  
That is, $\fh_k(T)=0$
 if $T$ is nilpotent of order $N.$ Equivalently, when writing
\[
 \fh_k(x) =\sum_{\alpha\in \sW^j} c_\alpha x^\alpha,
\]
 we have $c_\alpha=0$ whenever $|\alpha|<N.$
 The case of $N=2$ follows
 from Lemma~\ref{l:eees}. Now suppose the result is
 true for an $N\ge 2.$ 
 Let $T_k$ denote the $(N+1)\times (N+1)$ matrix with
 $\lambda_{k,1}=1$ in the $(1,2)$ entry; $\lambda_{k,\au}\in\CC\setminus \{0\}$ in the $(\au,\au+1)$
 entry for $\au\ge  2$ and $0$ elsewhere. Similarly, for $\ell\ne k,$
 let $T_\ell$ denote the $(N+1)\times (N+1)$ matrix with 
 $\lambda_{\ell,\au}\in \CC$ in the $(\au,\au+1)$ entry for $1\le \au \le N$
 and $0$ elsewhere. 
  The tuple $T=T(\lambda)$ thus depends (linearly) on 
 the $\lambda_{\ell,\au}\in\CC$ for $1\le \ell\le \vg$ 
 and $1\le \au\le N$  and $(\ell,\au)\ne (k,1).$ 
 Given such a $\lambda$ and a word $\alpha=x_{j_1}x_{j_2}\cdots x_{j_N},$
 let
\[
 \lambda^\alpha = \lambda_{j_1,1} \lambda_{j_2,2}\cdots \lambda_{j_N,N}
\]
and note that any $\lambda_{j,u}$ appears at most once
  in this product. Hence,
if  $c_\alpha\in \CC$ and
\[
 \sum_{|\alpha|=N} c_\alpha \lambda^\alpha =0,
\]
 for $\lambda$ an open set of such $\lambda$ (with $\lambda_{k,1}=1$),
 then $c_\alpha=0$ for all $\alpha.$

 In the case of item~\eqref{i:m+1}, there is an open 
 neighborhood $U$ of $0$ in $\CC^{N\vg-1}$ such that for
 $\lambda\in U,$ the tuple $T(\lambda)$ is in the boundary
 of $\fP_A$  by  Lemma~\ref{l:KRapplied}.  
 Let $M(\lambda) = \varphi_j(T(\lambda)).$ By the induction
 hypothesis, and using notation from equation~\eqref{e:vphj},
\begin{equation*}
 M(\lambda) = \fg_j(T_k(\lambda)) 
   \, + \sum_{|\alpha|=N, \alpha\ne x_k^N} a^k_\alpha T^\alpha(\lambda).
\end{equation*}
 Since, for $|\alpha|=N,$ the matrix $T^\alpha$ is
 $0$ except for its upper right entry.
 When $\lambda_{k,j}\ne 0,$  Lemma~\ref{l:evenmoreCint}
 implies 
\[
 \sum_{|\alpha|=N, \alpha\ne x_k^N} a^k_\alpha T^\alpha =0.
\]
 On the other hand,  
 letting $P$ denote the matrix with $1$ in the upper right
  entry and $0$ elsewhere, 
\begin{equation*}
\sum_{|\alpha|=N, \alpha\ne x_k^N} a^k_\alpha T^\alpha
  = \sum a^k_\alpha \lambda^\alpha \, P.
\end{equation*}
Hence 
 $a^k_\alpha=0$ for each $\alpha\in \sW^k_N.$
 Thus $\varphi_j(T) =\fg_j(T),$ completing the induction 
 argument in the setting of item~\eqref{i:m+1}.

 We now drop the $k\in \fN$ assumption and 
 proceeds as above, but with 
 $\lambda_{k-1,u}=0=\lambda_{k+1,u}$ for each $u.$ %
 The conclusion in this setting becomes:  if $\alpha$ is
 a word that does not include either $x_{k-1}$ and  $x_{k+1}$ 
 (and $\alpha\ne x_k^N$), 
 then $c_\alpha=0.$  Hence if $T\in\fP_A$ and 
 $T_{k-1}=0=T_{k+1},$ then $\fh_k(T)=0.$

\subsection{The sets $\fN$ and $\fZ$}
 This subsection develops criteria for  membership in
 the sets  $\fN$  and $\fZ,$ where
\begin{equation}
 \label{e:Zagain}
 \fZ=\{j: L(\delta_j +\epsilon\delta_{j+1}\not\succeq 0, \ \
   L(\epsilon \delta_{j-1}+\delta_j)\not \succeq 0, 
 \ \ \mbox{ for all } \epsilon>0\}
\end{equation}
 and $\fN=\{1,2,\dots,\vg\}\setminus \fZ.$

\begin{lemma}
\label{l:if-}
 Fix $1\le j< \vg$ and  let $k=\pi(j)$ and $\ell=\pi(j+1).$
 If  $L_A(\delta_j+\delta_{j+1})$  is not positive
 semidefinite, then either $\ell=k-1$ or $\ell=k+1.$

\end{lemma}

\begin{proof} 
  Arguing the contrapositive,  suppose $\ell\ne k+1$
 and $\ell\ne k-1.$  By Lemma~\ref{l:mobius+} item~\eqref{i:m+4},
 $\varphi_j(x\delta_k + w\delta_\ell) = \fg_j(x)$
 and $\varphi_{j+1}(x\delta_k+w\delta_\ell) = \fg_{j+1}(w),$
 in the latter case since $\ell=\pi(j+1)$ and $k\ne \ell+1$
 and $k\ne \ell-1.$ 
 On the other hand $t(\delta_k+\delta_\ell)$ is in 
  $\fP_A$ for $0<t<1$ and hence so is
 $\varphi(t(\delta_k+\delta_\ell)).$  In particular, 
\begin{equation}
 \label{e:if-1}
 \begin{pmatrix} I & \fg_{b_j}(t) \otimes C_j & 0 \\
  \fg_{b_j}(t)^*\otimes C_j^* & I & \fg_{b_{j+1}}(t)  \otimes C_{j+1} \\
  0 & \fg_{b_{j+1}}(t)^*  \otimes C_{j+1}^* & I \end{pmatrix} \succeq 0.
\end{equation}
 Because $\fg_{b_j}$ and $\fg_{b_{j+1}}$
  extend continuously across the boundary of the unit disc,
 it follows that the inequality of equation~\eqref{e:if-1} holds for $t=1.$
 Since $\fg_{b_j}(1)$ and $\fg_{b_{j+1}}(1)$ both have modulus $1,$
 it follows that $L_A(\delta_j+\delta_{j+1})\succeq 0$ 
 and the proof is complete.
\end{proof}

\begin{lemma}
 \label{l:if+}
 Fix $1\le j\le \vg$ and  let $k=\pi(j).$ %

\begin{enumerate}[(a)] \itemsep=7pt
 \item \label{i:if++}
 If $j\in\fZ^+,$ then 
  \begin{enumerate}[(i)]
   \item \label{i:if++i}  $\pi(j+1)=k+1$ 
   \item \label{i:if++ii}  $b_j=0=b_{j+1};$ and
   \item \label{i:if++iii} $k\in\fZ^+.$
 \end{enumerate}
  \item \label{i:if+-}  If $j\in \fZ^-,$ then 
      \begin{enumerate}[(i)]
   \item $\pi(j-1)=k-1;$ 
   \item $b_j=0=b_{j-1};$ and
   \item $k\in\fZ^{-}.$
  \end{enumerate}
\end{enumerate}
\end{lemma}

\begin{proof}
 Suppose $j\in \fZ^+.$ Thus, for each $\epsilon>0,$ the matrix
 $L_A(\delta_j+\epsilon \delta_{j+1})$ is not positive
 semidefinite. In particular,  if $\epsilon>0,$ then
\[
 C_j^*C_j + \epsilon^2  C_{j+1}C_{j+1}^* \not\preceq I,
\]
by Lemma~\ref{l:fZ+}.

 Choose 
\[
 T_k=S=\begin{pmatrix} 0&1\\0&0\end{pmatrix}
\]
 and $T_a=0$ otherwise. Thus $T$ is in the boundary of
 $\fP_A,$ the matrix $\varphi_j(T)=\fg_j(T_k)$ has norm one and 
 there is a scalar $e_{j+1}$ such that 
\[
 \varphi_{j+1}(T)=b_{j+1}I_2 + e_{j+1}T_{\pi(j+1)} 
  = b_{j+1}I. 
\]
 
On the other hand,  $L_A(\varphi(T))\succeq0$ implies,
 by Lemma~\ref{l:fZ+},
\[
 \varphi_j(T)^*\varphi_j(T)\otimes C_j^*C_j
 + |b_{j+1}|^2 C_{j+1}C_{j+1}^* \preceq I,
\]
 and thus $b_{j+1}=0$ by Lemma~\ref{l:fZ+}, proving
 half of item~\eqref{i:if++ii}.

 From Lemma~\ref{l:if-}, $\ell=k-1$ or $\ell=k+1.$ Arguing
 by contradiction, suppose $\ell=k-1.$  In this case,
 let $T_a=0$ for $k-1\ne a\ne k$ and 
\[
 T_k = \begin{pmatrix} 0 & 1 & 0\\0&0&0\\ 0&0&0\end{pmatrix},
 \ \ \
T_{k-1}=\begin{pmatrix} 0&0&0\\0&0&1\\0&0&0\end{pmatrix}.
\]
 Thus $T$ is in the boundary of $\fP_A$ and hence 
 so is $\varphi(T).$ The only word in $T$ of length two
 or longer that is not $0$ is  $T_k T_{k-1}.$ Hence,
 letting $e_j=e^{i\theta_j} (|b_j|^2-1)$ and $e_{j+1} = e^{i\theta_{j+1}},$
\[
\begin{split}
 \varphi_j(T) & = \fg_j(T_k) + \alpha T_k T_{k-1} 
  = b_j + e_j T_k + \alpha T_k T_{k-1} = 
   \begin{pmatrix} b_j & e_j & \alpha \\0& b_j &0\\0&0&b_j\end{pmatrix} \\
 \varphi_{j+1}(T) & =\fg_{j+1}(T_{k-1}) +\beta T_k T_{k-1}
  =  e_{j+1}T_{k-1} + \beta T_k T_{k-1}= 
 \begin{pmatrix} 0&0&\beta\\ 0&0& e_{j+1}\\
 0&0&0 \end{pmatrix},
\end{split}
\]
 for some $\alpha,\beta\in \CC.$  From Lemma~\ref{l:evenmoreCint},
 $\alpha=0=\beta.$  By Lemma~\ref{l:Cintstart} 
 there exists a unit vector $\gamma\in \CC^3$
 such that $\|\varphi_j(T)\gamma\|=1$ and 
  $\gamma_2\ne 0$ but $\gamma_3=0.$ 
 Since $|e_{j+1}|=1$ and $\gamma_2\ne 0,$ it follows that 
 $\varphi_{j+1}(T)^* \gamma\ne 0.$ 
 Since $L_A(\varphi(T))\succeq0,$ 
\[
 \begin{pmatrix} I & \varphi_j(T)\otimes C_j & 0 \\ 
  \varphi_j(T)^* \otimes C_j^*  & I & 
  \varphi_{j+1}(T)\otimes C_{j+1}  \\ 0 & 
   \varphi_{j+1}(T)^* \otimes C_{j+1}^*  & I \end{pmatrix} \succeq 0
\]
 and hence, by Lemma~\ref{l:fZ+},
\[
 I \succeq \|\varphi_j(T)\gamma\|^2 C_j^* C_j 
   + \|\varphi_{j+1}(T)^*\gamma\|^2   C_{j+1}C_{j+1}^* 
  = C_j^* C_j 
   + \|\varphi_{j+1}(T)^*\gamma\|^2   C_{j+1}C_{j+1}^*.
\]
 Lemma~\ref{l:fZ+} now gives the contradiction $\varphi_{j+1}(T)^*\gamma=0.$
 Thus $\ell=k+1,$ proving item~\eqref{i:if++i}.

 Next choose $T_k=S=T_{k+1}$ and $T_a=0$ otherwise,
 where $S$ is given in equation~\eqref{d:firstS}.
 Hence $T$ is in the boundary of $\fP_A$ and thus
 so is $\varphi(T).$ Letting $X=\varphi_j(T)$
 and $Y=\varphi_{j+1}(T),$ it follows that 
\[
 X^*X  \otimes C_j^*C_j + YY^* \otimes C_{j+1}C_{j+1}^*\preceq I.
\]
 Letting $\gamma$ denote a unit vector with
 $X^*X\gamma=\gamma,$ we find $Y^*\gamma=0.$ Since
 $b_{j+1}=0,$
\[
 Y=\begin{pmatrix} 0& e_{j+1} \\0&0\end{pmatrix},
\]
 where $|e_{j+1}|=1.$  It follows that the first
 component of $\gamma$ is zero and thus $b_j=0,$
 and the proof of item~\eqref{i:if++ii} is complete.
 
 To prove item~\eqref{i:if++iii}, suppose,
 by way of contradiction,  there is an $\epsilon>0$ such that
 $L_A(\delta_k+\epsilon \delta_{k+1})\succeq 0.$ 
 The tuple $T$ defined by $T_a=0$
 for $k\ne a \ne k+1$ and
\[
 T_k =\begin{pmatrix} 0&1&0\\0&0&1\\0&0&0\end{pmatrix}, \ \ \
 T_{k+1} =\begin{pmatrix} 0&1&0\\0&0&\epsilon\\0&0&0\end{pmatrix}
\]
is in the boundary of $\fP_A.$ Thus $\varphi(T)$ is in the 
 boundary of $\fP_A.$ Since $b_j=0,$ 
\[
 X=\varphi_j(T) e^{i\theta_j} T_k + \fh_k(T)
   = \begin{pmatrix} 0 & e^{i\theta_j} & y\\
 0&0&e^{i\theta_j}\\0&0&0\end{pmatrix},
\]
 for some $y.$ Since $\varphi_j(T)$ is a contraction,
 $y=0.$ Similarly, since $b_{j+1}=0,$ 
\[
 Y=\varphi_{j+1}(T) =\begin{pmatrix} 0&e^{i\theta_{j+1}} & 0\\
 0 & 0& \epsilon e^{i\theta_{j+1}} \\0&0&0 \end{pmatrix}.
\]
  Since $L_A(\varphi(T))\succeq 0$, an application
 of Lemma~\ref{l:fZ+} gives
\[
 X^*X\otimes C_j^*C_j + YY^* \otimes C_{j+1}C_{j+1}^* \preceq I,
\]
 from which it follows that
\[
 C_j^*C_j + \epsilon^2 C_{j+1}C_{j+1}^* \preceq I.
\]
 Equivalently $L_A(\delta_j+\epsilon \delta_{j+1})\succeq 0,$
  a contradiction that completes the proof of item~\eqref{i:if++}.

 Item~\eqref{i:if+-} is proved similarly.
\end{proof}

\subsection{Summary}
\label{s:mobius-sumry}
The following proposition records the key findings from this
section.

\begin{proposition}
\label{p:sumry}
  The sets $\fN$ and $\fZ$ partition  $\{1,2,\dots,\vg\}.$

  If $\varphi$ is an automorphism of $\fP_A,$ then 
  there is a permutation $\pi$ of $\{1,2,\dots,\vg\}$ and a tuple
 $\theta\in \RR^{\vg},$  such that 
\begin{enumerate}[(i)]
 \item the sets $\fN,$ $\fZ,$ $\fZ^+,$ $\fZ^-$ 
 are invariant under $\pi;$ 
 \item \label{i:sumryii} 
   if $j\in \fZ$ and $k=\pi(j),$  then there exists
  $\fh_k\in \fH^k$ such that 
\begin{equation}
\label{e:sumZ}
   \varphi_j(x)=e^{i\theta_j}x_{\pi(j)} +\fh_k(x),
\end{equation}
 and moreover, $\fh_k(T)=0$
 when $T_{k-1}=0=T_{k+1};$ 
\item if $j\in \fN$ and $k=\pi(j),$  then there is  an automorphism $\fg_j$
 of $\mathbb D$ such that
\begin{equation}
\label{e:sumN}
 \varphi_j(x) =\fm_{b_j}(e^{i\theta_j}x_{k})=\fg_j(x_k) = 
    \frac{b_j+e^{i\theta_j}x_k}{1+b_j^* e^{i\theta_j} x_k},
\end{equation}
 where $b_j=\varphi_j(0).$
\end{enumerate}
\end{proposition}

\begin{proof}
  By definition, $\fN$ and $\fZ$ partition.  
 
  Part of the conclusion of Lemma~\ref{l:if+} is that
  $\fZ^+$ and $\fZ^-$ are invariant under $\pi,$ from
 which it follows that $\fZ$ and $\fN$
 are also invariant. 

 If $j\in \fZ$ then, by Lemma~\ref{l:if+},
  $\varphi_j(0)=0$ and thus, 
 from  Lemma~\ref{l:mobius+} item~\eqref{i:m+4}, there is a
 $\theta_j\in\RR$ and $\fh_k\in\fH^k$ such that
 equation~\eqref{e:sumZ} holds.

 Finally, if $j\in \fN,$ then equation~\eqref{e:sumN} 
  holds by Lemma~\ref{l:mobius+} item~\eqref{i:m+1}.
\end{proof}

\section{\Norml Automorphisms and Compositions}
  Recall that $\fZ,\fN\subseteq \{1,2,\dots,\vg\}$ 
 (see equation~\eqref{e:Zagain}) depend
 only on $\fP_A$ and are thus independent
 of any particular automorphism of $\fP_A.$

 Suppose $\psi$ is an automorphism of $\fP_A$
 with associated permutation $\kappa.$  In particular
 the conclusions of Proposition~\ref{p:sumry} hold.
  We say that
 $\psi$ is a \df{\norml automorphism} if
\begin{enumerate}[(i)]
 \item \label{i:nai}
      $c_j=\psi_j(0)\ge 0;$ for each $j;$
 \item \label{i:naii}
    for $j\in \fZ$  and $k=\kappa(j),$
\[
  \psi_j(x) = \gamma_j x_{k} + \fh_k(x),
\]
 where  $\fh_k\in \fH^k$ and $\gamma_j\in \TT;$
 \item \label{i:naiii}
  for $j\in \fN$ and $k=\kappa(j),$ \index{$\fhm$}
\[
 \psi_j(x) = \fhm_{c_j}(x_{k}) :=\frac{c_j+ x_k}{1+ c_j x_k}.
\]
\end{enumerate}

\begin{lemma}
\label{l:make+}
 If $\varphi$ is an automorphism of $\fP_A$ with 
 associated permutation $\pi,$  then 
  there is a \norml  automorphism $\psi$ of $\fP_A$
 with associated permutation $\pi$ such that
 $\psi_j(0)=|\varphi_j(0)|$ for each $j.$ 
\end{lemma}

\begin{proof}
 Let $b_j=\varphi_j(0).$
From  Proposition~\ref{p:sumry}, 
 there is a permutation
 $\pi$ and a tuple $\theta \in \RR^{\vg}$  such that,
 for $j\in\fZ,$ 
\[
 \varphi_j(x) = e^{i\theta_j} x_{k} + \fh_k(x);
\] 
 and,  for $j\in \fN,$ 
\[
 \varphi_j(x) 
 = \fm_{b_j}(e^{i\theta_j}x_{\pi(j)}).%
\]

 Let  $\gamma_j=1$ if $b_j=0$ and $\gamma_j =\frac{b_j^*}{|b_j|}$
 otherwise. Let $\rho:\fP_A\to\fP_A$ denote the
 trivial  automorphism
\[
 \rho(x)=\gamma\bcdot x.
\]
Let  $\alpha_k = e^{-i\theta_{\pi^{-1}(k)}}\gamma_{\pi^{-1}(k)}^*$
and 
$\tau:\fP_A\to\fP_A$ denote the trivial  automorphism defined by
\[
\tau(x)=\alpha\bcdot x.
\]
 
 Finally, let $\psi:\fP_A\to\fP_A$ denote the automorphism
 $\psi = \rho \circ \varphi \circ \tau$ and note the 
 conclusion of the lemma is satisfied with $\kappa=\pi.$
 Indeed, with $k=\pi(j),$  if $j\in \fN,$ then,
\[
\begin{split}
 \psi_j(x)= & \gamma_j \varphi_j(\alpha\bcdot x) 
 =  \gamma_j\fm_{b_j}(e^{i\theta_j} (\alpha \bcdot x)_k)\\
  = &\gamma_j \fm_{b_j}(\gamma_j^* x_k) 
  = \gamma_j \frac{b_j + \gamma_j^* x_k}{1+ |b_j|x_k}\\
 &  = \frac{|b_j| + x_k}{1+|b_j|x_k}
   =\fhm_{|b_j|}(x_k). 
\end{split}
\]
 Hence item~\eqref{i:naiii} holds and moreover
  $\psi_j(0)=|b_j|\ge 0.$

If instead $j\in \fZ,$ then, since $\psi$ is an automorphism,
 then item~\eqref{i:naii} holds by item~\eqref{i:sumryii} of 
  Proposition~\ref{p:sumry}. Further $\psi_j(0)=0$
 and hence item~\eqref{i:nai} holds.
\end{proof}

\begin{lemma}
 \label{l:piplayswell}
 Suppose $\varphi$ and  $\wtvphi$ are automorphisms of $\fP_A$
 with associated permutations $\pi$ and $\wtpi.$
 Let $\psi$ denote the automorphism  
 $\wtvphi \circ \varphi$ and let $\tau$ 
 be its associated permutation.  If $j\in \fN,$
 then   $\tau(j)=\pi(\wtpi(j))$ and moreover, 
 with  $\wtvphi_j(x)=\widetilde{g_j}(x_{\wtpi(j)})$
 and $\varphi_{\wtpi(j)}(x)=\fg_{\wtpi(j)}(x_{\pi(\wtpi(j)}),$
 we have 
$\psi_j(x) =\widetilde{g_j} \circ \fg_{\widetilde{\pi}(j)}(x_{\tau(j)}).$
\end{lemma}

\begin{proof}
 Let $\wtb=\wtvphi(0)$ and $b=\varphi(0).$
  Given $j\in \fN,$ it follows that 
 $\pi(j)\in \fN$ and  $\pi(\wtpi(j))\in \fN$
 by Proposition~\ref{p:sumry}. 
 From item~\eqref{i:m+1} of Lemma~\ref{l:mobius+},
 there exists functions $g_j$ and $\widetilde{g}_j$
 as well as $\fg_{\wtpi(j)}$ 
 of one variable such that 
  $\wtvphi_j(x) = \widetilde{g_j}(x_{\wtpi(j)})$ and
 $\psi_j(x) = g_j(x_{\tau(j)})$ as well 
 as $\varphi_{\wtpi(j)}(x)=\fg_{\wtpi(j)}(x_{\pi(\wtpi(j)}).$ Hence
\[
 g_j(x_{\tau(j)}) 
   =  \psi_j(x) = \wtvphi_j(\varphi(x))
   =  \widetilde{g}_j(\varphi_{\wtpi(j)}(x))
  = \widetilde{g}_j(\fg_{\wtpi(j)}(x_{\pi(\wtpi(j)}))
 = \widetilde{g}_j\circ \fg_{\wtpi(j)} (x_{\pi(\wtpi(j))}).
\]
 It follows that $\tau(j)=\pi(\wtpi(j))$ 
 and $g_j(y) =\widetilde{g_j} \circ \fg_{\widetilde{\pi}(j)}(y).$
\end{proof}

Given an automorphism $\varphi$ of $\fP_A,$  let \index{$\cF_\varphi$}
\[
 \cF_{\varphi}=\{j: \varphi_j(0)\ne 0\})
\]
 and note that $\cF_\varphi\subseteq \fN$ by Proposition~\ref{p:sumry}.

\begin{lemma}
\label{l:compose}
 If $\varphi$ and $\wtvphi$ are  \norml automorphisms
 of $\fP_A$  with associated permutations $\pi$  and 
 $\wtpi$ respectively,   then $\psi=\wtvphi\circ\varphi$ is
 a \norml automorphism of $\fP_A$ and moreover, 
 $\cF_\psi=\cF_{\wtvphi}\cup \wtpi^{-1}(\cF_{\varphi}).$
\end{lemma}

\begin{proof}
 Let $\tau$ denote the permutation associated with $\psi.$ 
 By Lemma~\ref{l:piplayswell}  $\tau=\pi \circ \wtpi$ on $\fN.$ 
  Let $\widetilde{b}=\wtvphi(0)$ and $b=\varphi(0).$
 By assumption $\widetilde{b}_j,\, b_j\ge 0.$
 
 Suppose $j\in \fN.$ From  Proposition~\ref{p:sumry}, 
 $\tau(j), \wtpi(j), \pi(j)\in \fN$ and 
\[
 \psi_j(x) = \wtvphi_j(\varphi(x)) 
  = \fhm_{\widetilde{b_j}}(\varphi_{\wtpi(j)}(x)) 
  = \fhm_{\widetilde{b_j}}(\fhm_{b_{\wtpi(j)}}(x_{\tau(j)}))
  = \fhm_{c_j}(x_{\tau(j)}),
\]
 where $c_j=(\widetilde{b_j}+b_{\wtpi(j)})
(1+\widetilde{b_j}b_{\wtpi(j)})^{-1}\ge 0.$ 
  In particular,  $\psi_j(0)\ge 0.$  Moreover, 
 $\psi_j(0)>0$    
  if and only  either $\widetilde{b_j}>0$ 
  or $b_{\wtpi(j)}>0.$
  Hence, if $j\in \fN,$ then   $j\in \cF_\psi$ if and only if
  either $j\in \cF_{\wtvphi}$
 or $\wtpi(j)\in \cF_{\varphi}.$  On the other hand,
 if $j\notin\fN,$ then  $\wtpi(j)\notin\fN$ and 
therefore $j\notin \cF_{\psi}\cup \cF_{\wtvphi}$ and
 $\wtpi(j)\notin \cF_{\varphi}.$
 Hence,  $\cF_\psi=\cF_{\wtvphi}\cup \wtpi^{-1}(\cF_{\varphi}).$ 
\end{proof}

Given a self map $f$ of a set $X$ and a positive  integer $n,$
  let \df{$f^{(n)}$} denote the composition of  $f$ with itself $n$-times.

\begin{lemma}
\label{p:sumry2}
 If $\varphi$ is a \norml automorphism of $\fP_A,$ then 
  there  is a positive integer $n$ such
 that $\psi=\varphi^{(n)}$ is a \norml automorphism
 of $\fP_A$ %
 such that, for each positive integer $m,$ 
\begin{enumerate}[(a)]
 \item \label{i:sy1}
   the permutation associated to $\psi^{(m)}$ 
   is the identity on $\fN;$ 
 \item \label{i:sy2}
 $\cF_\varphi \subseteq \cF_\psi=\cF_{\psi^{(m)}};$ and
 \item \label{i:sy3}
  both $\cF_\psi$ and $\widetilde{\cF_\psi}$ are invariant 
 for the permutations associated to $\psi^{(m)}.$
 \end{enumerate}
\end{lemma}

\begin{proof}
 By Lemma~\ref{l:compose}, for each positive integer $p$ we have $\varphi^{(p)}$ is a \norml
 automorphism of $\fP_A$ and 
\[
 \cF_{\varphi} \subseteq \cF_{\varphi^{(2)}} \subseteq \cF_{\varphi^{(3)}} \subseteq \cdots.
\]
 Hence, there is $N$ such that $\cF_{\varphi^{(p)}} =\cF_{\varphi^{(N)}}$ for all $p\ge N.$

 Let $\sigma$ denote the permutation
 associated to $\varphi^{(N)}.$ 
  By Proposition~\ref{p:sumry}, $\fN$ is invariant
 under $\sigma.$ Let $\rho=\sigma|_{\fN}:\fN\to\fN.$ 
 There is an $\ell\ge 1$ such that $\rho^{(\ell)}$ is the identity. 
 Let $n=N\ell$ and let $\psi=\varphi^{(N\ell)}.$ 
 Let $\pi$ denote the permutation associated to
 $\psi.$  By Lemma~\ref{l:piplayswell}, $\pi =\sigma^{(\ell)}$
 and in particular $\pi$ is the identity on $\fN.$
 Similar reasoning shows $\pi^{(m)}$ is the permutation
 associated with $\psi^{(m)}$ and that $\pi^{(m)}$ 
 is the identity on $\fN$ proving item~\ref{i:sy1}.
 Item~\eqref{i:sy2}  follows  
 Lemma~\ref{l:piplayswell}.

 To prove item~\eqref{i:sy3} note that, since 
 $\pi$ is the identity on $\fN$ and 
 $\cF_{\psi} \subseteq \fN,$ it follows
 that $\pi$ is the identity on $\cF_{\psi}$
 and thus both $\cF_{\psi}$ and (therefore) its
 complement are invariant for $\pi.$
\end{proof}

\section{The Case of the Identity Permutation}
\label{sec:identity}
 In this section, we fix an automorphism $\psi$ of $\fP_A$ 
 with associated permutation $\pi$ satisfying
 the conclusion of Proposition~\ref{p:sumry2} and assume
 $\cF_\psi\ne \varnothing.$ (In particular,
 $\fN\ne \varnothing$ by Proposition~\ref{p:sumry}.)  Thus
 $\psi$ is a \norml automorphism, $\pi$ is the identity
 on $\fN$ and $\cF_{\psi^{(m)}} =\cF_{\psi} \ne \varnothing$ for all positive
 integers $m.$  For notational convenience,
 let $\cF=\cF_\psi.$ \index{$\cF$} The permutation
 $\pi$ is the identity on $\cF$ and both $\cF$ and $\ccG$
 are invariant for $\pi.$
 Let $b=\psi(0).$ Thus $b_j=\psi_j(0)\ge 0$ for all $j$ and
 $b_j>0$ if and only if $j\in \cF.$

\subsection{Two auxiliary Reinhardt  spectrahedra}
\label{s:aux}
 Let \index{$M_n(\C)^{\cG}$} \index{$M_n(\C)^{\ccG}$}
\[
 M_n(\C)^{\cG}=\{X=(X_j)_{j\in \cG}: X_j\in M_n(\C)\}.
\]
 Similarly,  let
\[
 M_n(\C)^{\ccG}=\{Y=(Y_j)_{j\in\ccG}: Y_j\in M_n(\C)\}.
\]
Given $X\in M_n(\C)^{\cG}$ and $Y\in M_n(\C)^{\ccG},$ let
 $T=(X,Y)\in M_n(\C)^{\vg}$ denote the tuple with 
 $T_j=X_j$ for $j\in \cG$ and $T_j=Y_j$ for $j\in\ccG.$

 Let \index{$\hat{b}$}  $\hat{b}=(b_j I_n)_{j\in \cG}\in M_n(\C)^{\cG}.$
 (Thus the $n$ is understood from context.)
 Similarly, let \index{$\hat{0}$} $\hat{0}=(0 I_n)_{j\in \cG}.$ With these 
 notations, let
\begin{equation*}
\begin{split}
    \fP_A^\cG = & \{Y\in M(\C)^{\ccG}: (\hat{0},Y)\in \fP_A\} \\
   \mfE_A^{\cG} =& 
   \{Y\in M(\C)^{\ccG}: (\hat{b},Y)\in\fP_A\}.
\end{split}
\end{equation*}
 \index{$\fP_A^\cG$} \index{$\mfE_A^{\cG}$}
Since $\fP_A$ is Reinhardt, it immediately follows that $\mfE_A^{\cG}$ is
 a subset of $\fP_A^{\cG}.$ %

\begin{lemma}
\label{l:mfEisaspec}
 Both $\fP_A^{\cG}$ and   $\mfE_A^{\cG}$ are 
 Reinhardt free spectrahedra  in $\vg-|\cG|$ variables.
\end{lemma}

\begin{proof}
 That both domains are Reinhardt is immediate. It is also
 evident that $(\hat{0},Y)\in \fP_A$ if and only if
\[
 I + \sum_{j\in\ccG} A_j\otimes Y_j 
  +\sum_{j\in\ccG} A_j^*\otimes Y_j^* \succ0,
\]
 and hence 
 $\fP_A^{\cG}$ is a free spectrahedron.

 Let
\[
 P = I+\sum_{j\in \cG} A_j b_j +\sum_{j\in\cG} A_j^*b_j^*,
\]
  and note $Y\in \mfE_A^{\cG}$ if and only if
  $(\hat{b},Y)\in\fP_A$  if and only if 
\[
 P\otimes I
+ \sum_{j\notin\cG} A_j\otimes Y_j 
  +\sum_{j\notin\cG} A_j^*\otimes Y_j^* \succ0.
\]
For $j\in \ccG,$ let $B_j = P^{-\frac 12} A_j P^{-\frac12}$
 and observe  $Y\in \mfE_A^{\cG}$ if and only if
\[
 I + \sum_{j\notin\cG} B_j\otimes Y_j 
  +\sum_{j\notin\cG} B_j^*\otimes Y_j^* \succ0.
\]
 Hence $\mfE_A^{\cG}$ is a free spectrahedron too.
\end{proof}

\subsection{Two auxiliary automorphisms}
  Given  $Y\in \fP_A^{\cG},$ let $T=(\hat{0},Y)\in\fP_A$
and define \index{$\psi^{\cG}$}
\[
 \psi^{\cG}(Y)=(\psi_j(T))_{j\in\ccG}.
\]
 If  $j\in \cG,$  then, since 
 $\pi$ is the identity on $\cG;$
 the automorphism $\psi$ is normalized;
 and  $j\in \fN,$
\[
 \psi_j(T)= \fhm_{b_j}(T_{\pi(j)}) = \fhm_{b_j}(T_j)=\fhm_{b_j}(0)=b_j.
\]
 Thus
\[
 \psi(T)=\psi((\hat{0},Y))
   = (\hat{b},\psi^{\cG}(Y))\in\fP_A.
\]
 It follows that $\psi^{\cG}(Y)\in \mfE_A^{\cG}$ and thus
$\psi^{\cG}$ determines a free mapping
 $\psi^{\cG}:\fP_A^{\cG} \to \mfE_A^{\cG}.$

\begin{lemma}
\label{l:subbi}
 The mapping $\psi^{\cG}:\fP_A^{\cG}\to\mfE_A^{\cG}$ is bianalytic. 
\end{lemma}

\begin{proof}
 Evidently $\psi^{\cG}$ is analytic and one-one. It
 remains to show it is onto. To this end, let $Z\in \mfE_A^{\cG}$
 be given. Thus $(\hat{b},Z)\in\fP_A.$ Since $\psi$ is 
 onto, there is a $T\in \fP_A$ such that 
 $\psi(T)=(\hat{b},Z).$  For $j\in \cG\subseteq \fN,$ 
\[
 b_jI = \psi_j(T)=\fhm_{b_j}(T_j)=(b_j+T_j)(I+b_j T_j)^{-1}
\]
 and hence $T_j=0.$  Thus, setting $Y_j=T_j$ for $j\in \ccG,$
 it follows that $T=(\hat{0},Y),$ so that $Y\in \fP_A^{\cG},$
 and  $\psi((\hat{0},Y))=(\hat{b},Z).$ 
 Hence $Y\in \fP_A^{\cG}$ and $\psi^{\cG}(Y)=Z.$
\end{proof}

\begin{lemma}
\label{l:vpfZ}
 For each  $j\notin \cG$ there exists $\gamma_j\in \TT$ such that,
\[
 \psi_j(\hat{0},Y) = \gamma_j Y_{\pi(j)},
\]
  for all  $Y\in \fP_A^{\cG}.$  In particular,  
 for each $Y\in \fP_A^{\cG}$ and $j\notin\cF,$
\[
 \psi^{\cG}_j(Y) =  \gamma_j Y_{\pi(j)}.
\]
\end{lemma}

\begin{proof}
  The mapping $\psi^{\cG}:\fP_A^{\cG}\to \mfE_A^{\cG}$ is bianalytic by Lemma~\ref{l:subbi}.
 Moreover, $\psi^{\cG}(0)=0$ since $b_k=0$ for $k\in\ccG.$
 Since,  both $\fP_A^{\cG}$ and $\mfE_A^{\cG}$ 
  are circular  free spectrahedra (Lemma~\ref{l:mfEisaspec})
 and since $\psi^{\cG}(0)=0,$ 
 it follows from \cite[Theorem~4.4]{proper} that
 $\psi^{\cG}$ is linear.

 Now suppose $j\in \fZ.$ Since 
 $\psi$ is normalized, there exists $\gamma_j\in \TT$ 
 and $\fh_{\pi(j)}\in \fH^{\pi(j)}$ such that 
\[
 \psi_j(x) = \gamma_j x_{\pi(j)} + \fh_{\pi(j)}(x).
\]
 Thus, 
 for  $Y\in \fP_A^{\cG},$ 
\[
 \psi_j^{\cG}(Y) = \psi_j(\hat{0},Y)
   =  \gamma_j Y_{\pi(j)} + \fh_{\pi(j)}(\hat{0},Y).
\] 
 By linearity, it
  follows that $\fh_{\pi(j)}(\hat{0},Y)=0$
 for $Y\in \fP_A^{\cG}$ and 
 hence $\psi_j^\cF(Y)=\gamma_j Y_{\pi(j)}.$ 

 If $j\in \fN\cap \ccG$ and $Y\in\fP_A^{\cG},$ 
 then $\psi_j^{\cG}(Y) =\fhm_{b_j}(Y_j)=\fhm_0(Y_j)=Y_j,$
 since $\pi$ is the identity on $\fN.$ 
\end{proof}

Let $\iota$ denote the inclusion of $\mfE^{\cG}$ into $\fP_A^{\cG}$
 and $\newpsi^{\cG} = \iota\circ \psi^{\cG}:\fP_A^{\cG}\to \fP_A^{\cG}.$

\index{$\newpsi^{\cG}$}
\begin{lemma}
 \label{l:metoo}
   The mapping $\newpsi^{\cG}:\fP_A^{\cG}\to\fP_A^{\cG}$
  is bianalytic, and, for $j\notin\cG,$ there exists
 $\gamma_j\in \TT$ such that
\[
 \newpsi^{\cG}_j(T) = \gamma_j T_{\pi(j)}.
\]
\end{lemma}

\begin{proof}
For notational ease, let $\newpsi=\newpsi^{\cG}.$ Since
 both $\iota$ and $\psi^{\cG}$ are free analytic
 and injective, so is $\newpsi.$ It remains
 to prove $\newpsi$ is onto.

For each $j\notin \cF,$ there exist $\gamma_j\in \TT$ such that,
if $Y\in \fP_A^{\cG},$ then  
\[
 \newpsi_j(Y) = \psi^{\cG}_j(Y) =\gamma_j  Y_{\pi(j)}
\]
 by Lemma~\ref{l:vpfZ}. 
 Since, by Lemma~\ref{p:sumry2}, $\ccG$ is invariant for $\pi,$ there is a positive integer
$n$ such that $\pi^{(n)}$ is the identity on $\ccG$
($\pi$ was already the identity on $\cN\supseteq \cF$).
Thus, for $j\notin \cF,$ there exists $\delta_j\in \TT$ such that 
\[
\newpsi^{(n)}_j(Y) = \delta_j Y_j.
\]
Since the mapping $Y\mapsto \delta\bcdot Y,$ where $(\delta\bcdot Y)_j=\delta_jY_j,$
 is an automorphism of $\fP_A$ (Lemma~\ref{l:mfEisaspec})
  it follows that 
 $\newpsi^{(n)}$ is onto and thus so is $\newpsi.$ 
\end{proof}

\subsection{How to spot a polydisc}
\label{s:fI}
Recall the notation $(X,Y)\in \fP_A$ from the
first paragraph of subsection~\ref{s:aux}. 
Given $Y\in M_n(\C)^{\cG}$  and $\lambda\in \CC,$ 
let $Y^\lambda =(\fI^\lambda,Y) \in M_n(\C)^{\vg},$
 where $\fI^\lambda_j=\lambda I_n$ for $j\in \cF.$
\index{$\fI^\lambda$}

\begin{lemma}
 \label{l:somesum}
 If  $Y^\lambda \in \fP_A^{\cG}$  for each
  $0\le \lambda <1$ and 
   $Y \in \fP_A^{\cG},$ then   $\fP_A$ contains a polydisc
 as a distinguished summand.
\end{lemma}

\begin{proof}
 Let $T=(X,Y)$ such that $(\hat{0},Y)\in \fP_A$ and $\|X_j\|\le 1$
 for $j\in \cG$ be given. Thus $Y\in \fP_A^{\cG}.$ 
 Given $0\le \lambda <1,$ by assumption $Y^\lambda\in \fP_A.$
 An application of Lemma~\ref{l:KR++} implies
 $(\lambda X,Y)\in \fP_A.$ 
\end{proof}

\section{Proof of Theorem~\ref{t:main}}
 It remains to prove items~\eqref{i:main1} and \eqref{i:main3}
 of Theorem~\ref{t:main},
 since item~\eqref{i:main2} was already established
 as Proposition~\ref{p:main2}.

\subsection{Proof of item~\eqref{i:main1}}

\begin{proposition}
 \label{p:main}
 If $\fP_A$ does not contain a polydisc as a distinguished
  summand and if  $\varphi$ is an automorphism of  $\fP_A,$ then
 $\varphi(0)=0.$
\end{proposition}

\begin{proof}
 Arguing the contrapositive, suppose there is an automorphism
 of $\varphi$ with $\varphi(0)\ne 0.$ By Lemma~\ref{l:make+},
 we may assume $\varphi$ is normalized. By Lemma~\ref{p:sumry2},
 there is an automorphism  
 $\psi$ satisfying the conclusion of Lemma~\ref{p:sumry}
 and such that $\varnothing \ne \cF_\varphi \subseteq \cF_{\psi}.$ 
 In particular, for each positive integer $m,$ the automorphism
  $\psi^{(m)}$ is normalized; its associated permutation
 is the identity on $\fN;$ and
  $\varnothing \ne \cF_\varphi \subseteq \cF_\psi =\cF_{\psi^{(m)}}.$
 
 For notational ease, let $\cF=\cF_\psi.$
 Fix $j\in \cF.$ Since  $\psi$ is
 normalized;  $j\in \fN;$ and the permutation associated
 to $\psi$  is the identity on $\fN,$
\begin{equation*}
 \psi_j(x)= \fhm_{b_j}(x_j)=\frac{b_j+x_j}{1+b_j x_j},
\end{equation*}
 with $b_j> 0.$
 Now suppose $m$ is a positive integer and
  $\psi_j^{(m)}(0)=\fhm_{b_j}^{(m)}(0).$
 Since the permutation associated to both $\psi^{(m)}$ and $\psi$
 is 
 the identity on $\fN,$ an application of Lemma~\ref{l:piplayswell}
 gives
\[
 \psi_j^{(m+1)}(0)= \psi_j(\psi^{(m)}(0)) =
  \fhm_{b_j}(\varphi_j^{(m)}(0)) = \fhm_{b_j}(\fhm_{b_j}^{(m)}(0)).
\]
 Thus, by induction, $\varphi_j^{(m)}(0) = \fhm_{b_j}^{(m)}(0)$
 for all $m.$

 It is easily seen that the sequence $(\fhm_{b_j}^{(m)}(0))_m$
 converges to $1.$ (It is an increasing
 sequence from $(0,1)$ that cannot converge to an $L<1.$)
  Hence, given $0<\lambda <1,$ there is
 an $m$ such that 
\[
 \psi_j^{(m)}(0) \ge \lambda
\]
 for each $j\in\cF.$ Let $\rho=\psi^{(m)}.$ In particular,
 $\rho$ is a \norml automorphism, the permutation $\kappa$
 associated to $\rho$ is the identity on $\fN\supseteq \cF$
 and $\cF_\rho=\cF\ne \varnothing.$ 

 By Lemma~\ref{l:metoo}, given $Y\in \fP_A^{\cG},$ there exists
 a $Z\in \fP_A^{\cG}$ such that $\rho^{\cG}(Z)=Y.$ It follows
 that
\[
 \rho(\hat{0},Z)= (\hat{c},\rho^{\cG}(Z))=(\hat{c},Y),
\]
 where $c_j=\rho_j(0)\ge \lambda$ for $j\in \cF.$
 In particular, $(\fI^\lambda,Y)\in \fP_A,$
 where $\fI^\lambda$ is defined at the 
 outset of Subsection~\ref{s:fI}.  Thus,  by 
 Lemma~\ref{l:somesum}, $\fP_A$ contains a polydisc
 as a distinguished summand.
\end{proof}

\subsection{Proof of item~\eqref{i:main3}}

 Item~\eqref{i:main3} of Theorem~\ref{t:main} follows
 from Proposition~\ref{p:nopi} below.

\begin{proposition}
 \label{p:nopi}
 If there exists non-identity 
 permutation  $\rho$ of $\{1,2,\dots,\vg\}$ such that
 the mapping $\psi_j(x)=x_{\rho(j)}$ is an automorphism 
 of $\fP_A,$ then $\fP_A$ is a coordinate direct
 sum.
\end{proposition}

 Before proving Proposition~\ref{p:nopi}, we present
 two lemmas.
 Let $\{\ee_1,\ee_2\}$ denote the standard basis for
 $\C^2.$ Thus $\ee_j \ee_k^*$ are the usual matrix units
 for $M_2(\C).$

\begin{lemma}
 \label{l:nopi+}
   Fix $1\le \mu \le \vg-1$ and let $X\in M_n(\C)^{\vg}$
 be given. 
 If  $(X_1,\dots,X_{\mu},0,\dots,0),$ and $(0,\dots,0,X_{\mu+1},\dots,X_{\vg})$ are both
 in $\fP_A$ and if  $\ell_j\in \{1,2\}$
 for $j>\mu+1,$ then  the tuple $Z$ defined by
\begin{enumerate}[(a)]\itemsep=7pt
 \item  $Z_j=\ee_1\ee_1^* \otimes X_j$ for  $1\le j<\mu;$
 \item $Z_\mu=\ee_1\ee_2^* \otimes X_\mu;$
 \item $Z_{\mu+1} = \ee_1\ee_1^* \otimes  X_{\mu+1};$
 \item $Z_j=\ee_{\ell_j} \ee_{\ell_j}^* \otimes X_j$ for $j>\mu+1,$
\end{enumerate}
 is in $\fP_A.$
\end{lemma}

 \begin{proof}
Let 
\begin{enumerate}[(a)] \itemsep=7pt
\item $Y_\mu = \ee_1\ee_1^* \otimes X_\mu;$ 
 \item $Y_{\mu+1} = \ee_2\ee_2^* \otimes X_{\mu+1};$ and
 \item $Y_j= I_2\otimes X_j$ for $\mu+1\ne j \ne \mu.$
\end{enumerate}
 By hypothesis, the matrices
\begin{equation*}
L_{s,t} \begin{pmatrix} I & C_s\otimes X_s & 0 & 0 & \dots &0 &0 \\
        C_s^*\otimes T_1^* & I & C_{s+1}\otimes X_{s+1} & 0 & \dots &0 & 0 \\
        0 & C_{s+1}^*\otimes X_{s+1}^*  & I & C_{s+2}\otimes X_{s+2} & \dots &0 & 0 \\
       0&0& C_{s+2}^*\otimes X_{s+2}^*&I&\dots &0&0\\
        \vdots &\vdots&\vdots& \vdots& \dots &  \vdots \\
       0&0&0&0 &\dots & I & C_{t}\otimes T_t \\
       0&0&0&0 &\dots &C_t^* \otimes T_t^* &  I 
 \end{pmatrix}
\end{equation*}
 are positive semidefinite for $(s,t)\in \{(1,\mu),(\mu+1,\vg), (\mu+2,\vg).$
 Since $L_A(Y)$ is a direct sum of these matrices with an identity matrix
 (of an appropriate size),  $L_A(Y)\succeq0.$ Equivalently, 
 $Y\in \fP_A.$ 

Let 
\[
 W_j = \begin{pmatrix} 0&I\\I&0 \end{pmatrix}
\]
for $j\ge \mu$ and $W_j=I$ for $0\le j< \mu$ and observe 
 that, in the notation of equation~\eqref{d:AST}, 
\begin{enumerate}[(i)]
 \item  $(W\AST Y)_\mu= \ee_1 \ee_2^* \otimes X_\mu;$
 \item  $(W\AST Y)_{\mu+1} = \ee_1\ee_1^* \otimes X_{\mu+1};$ and
 \item  $(W\AST Y)_j=I_2\otimes Y_j$ otherwise.
\end{enumerate}
 By Lemma~\ref{l:U's}, $W\AST Y\in \fP_A,$ from which it readily follows
 that $Z\in \fP_A.$
\end{proof}

\begin{lemma}
 \label{l:nopi-}
  Fix $2\le \nu \le \vg.$ 
   Given a tuple $T\in M(\C)^{\vg},$ if the   tuple 
 $Z_j$ defined by 
\begin{enumerate}[(i)]
 \item  $Z_j=\ee_1\ee_1^* \otimes T_j$ for $j<\nu;$
 \item $Z_\nu = \ee_1 \ee_2^*\otimes T_\nu$  for $j=\nu;$ and
 \item $Z_j= \ee_2 \ee_2^* \otimes T_j$ for $j>\nu$
\end{enumerate}
 is in $\fP_A,$ then $T\in\fP_A.$  
\end{lemma}

\begin{proof}
 The proof of this lemma is similar to that 
 of Lemma~\ref{l:nopi+} and is omitted.
\end{proof}

\begin{proof}[Proof of Proposition~\ref{p:nopi}]
 Let $\mu$ denote the first index $\mu$ such that 
 $\rho^{-1}(\mu+1)<\rho^{-1}(\mu).$
 Such an index must exist since $\rho$ is not the identity
 permutation. Hence 
 $\rho^{-1}(\mu)>  \rho^{-1}(\mu-1) > \dots >\rho^{-1}(1);$
 in particular $\rho^{-1}(j) < \rho^{-1}(\mu)$ for $j< \mu.$

 Let $X\in M_n(\C)^{\vg}$
 be given and suppose $(X_1,\dots,X_{\mu},0,\dots,0),$
 and $(0,\dots,0,X_{\mu+1},\dots,X_{\vg})$ are both
 in $\fP_A.$   Define 
 a tuple $Z$ as follows:
\[
 Z_j=\begin{cases} \ee_1\ee_2^* \otimes X_\mu & \mbox{ for } j=\mu \\
      \ee_1\ee_1^*\otimes X_j  & \mbox{ for }
        \rho^{-1}(j)<\rho^{-1}(\mu) \\
      \ee_2\ee_2^* \otimes X_j & \mbox{ for }
               \rho^{-1}(j)>\rho^{-1}(\mu). \end{cases} 
\]
 Let $\nu=\rho^{-1}(\mu).$
 Since $\rho^{-1}(j)<\nu$ for $j<\mu$ and 
 $j=\mu+1,$ 
 Lemma~\ref{l:nopi+} implies $Z\in \fP_A.$ Thus $S=\psi(Z)\in \fP_A$
 and $S_k=Z_{\rho(k)}.$ 
 If $k<\nu$ and $k=\rho^{-1}(j),$ then 
 $S_k=Z_j=\ee_1\ee_1^*\otimes X_j.$ 
 If $k>\nu$ and $k=\rho^{-1}(j),$  then $S_k=Z_j=\ee_2\ee_2^*\otimes X_j.$
 Finally, if $k=\nu,$ then $S_k = Z_{\rho(k)} =Z_\mu
 =\ee_1\ee_2^* \otimes X_\mu.$
 An application of Lemma~\ref{l:nopi-} implies 
 the tuple $T$ with $T_k=X_{\rho(k)}$ is in $\fP_A;$
 that is $\psi(X)\in \fP_A.$ Hence $\psi^{-1}(\psi(X))=X\in \fP_A.$
 Thus $\fP_A$ is a coordinate direct sum.
\end{proof}


\addresseshere

\printindex


\begin{thebibliography}{99}


\bibitem[AHKM18]{longmaps}
Meric Augat, Igor Klep, Bill Helton and Scott McCullough,
 {\it Bianalytic maps between free spectrahedra,}
 Math. Annalen 371 (2018), no. 1-2, 883--959.

\bibitem[BPT13]{BPT}
Grigoriy Blekherman, Pablo Parrilo and Rekha Thomas (editors),
{\it  Semidefinite optimization and convex algebraic
geometry,} MPS-SIAM Series on Optimization, Series Number 13, 2013.

\bibitem[dOH06]{convert-to-matin}
 M.C. De Oliveira and  J. W. Helton, 
{\it Computer algebra tailored to matrix inequalities incontrol}
 International Journal of Control, (2006)
 79:11, 1382-1400, DOI:10.1080/00207170600725529

\bibitem[dOHMP09]{emerge}
  Mauricio de Oliveira, J. William  Helton, 
 Scott  McCullough and Mihai  Putinar,
 {\it  Engineering systems and free semi-algebraic geometry} in 
  Emerging applications of algebraic geometry, 17--61, IMA 
  Vol. Math. Appl., 149, Springer, New York, 2009.

 \bibitem[EHKM17]{circular}
 Eric Evert, Bill Helton, Scott McCullough and Igor Klep,
 {\it Circular Free Spectrahedra,}  J. Math. Anal. Appl, 
  {\bf 445} (2017), no. 1, 1047--1070.



\bibitem[HKM11]{proper}
J. William Helton,  Igor Klep and Scott McCullough, 
{\it Proper analytic free maps,}
  J.  Functional Analysis 260 (2011), no. 5, 1476--1490.

 \bibitem[HKM12]{billvolume}
J. William Helton,  Igor Klep and Scott McCullough, 
{\it Free analysis, convexity and LMI domains.} 
 Mathematical methods in systems, optimization, and control, 195–219,
Oper. Theory Adv. Appl., 222, Birkhäuser/Springer Basel AG, Basel, 2012.



\bibitem[HKMV20]{bestmaps} 
 J. William Helton,  Igor Klep, Scott McCullough and Jurij Vol\v ci\v c,
  {\it Bianalytic free maps between spectrahedra and spectraballs,} 
    Journal of Functional Analysis, {\bf 278} no. 11, 15 June 2020, 
   doi.org/10.1016/j.jfa.2020.108472.

\bibitem[HM09]{annals}
  J. William Helton and Scott McCullough, 
 {\it Every convex free basic semi-algebraic set has an LMI representation,} 
Annals of Mathematics (2) 176 (2012), no. 2, 979--1013.

\bibitem[HJ13]{HJ}
 Roger Horn and Charles Johnson,
 {\it Matrix analysis,}
  Second edition. Cambridge University Press, Cambridge, 2013. 


\bibitem[JP08]{first-steps}
 Marek Jarnicki and Peter Pflug, 
{\it First steps in several complex variables: Reinhardt domains,}
EMS Textbooks in Mathematics. European Mathematical Society (EMS), Zürich, 2008. 


\bibitem[KVV14]{KVV}
 Dmitry Kalyuzhnyi-Verbovetski\u{i} and  Victor Vinnikov,
{\it Foundations of free noncommutative function theory}, 
Mathematical Surveys and Monographs 199, Amer. Math. Soc., 2014.


\bibitem[MT+]{MT}
 Scott McCullough and Nicole Tuovila, 
{\it Reinhardt Free Spectrahedra,}  Linear Algebra and 
 its Applications  640 (2022), 91--117. 


\bibitem[MT16]{MT16}
John  McCarthy, Richard Timoney:
\textit{Non-commutative automorphisms of bounded non-commutative domains},
Proc. Roy. Soc. Edinburgh Sect. A
{146} (2016) 1037--1045. 


\bibitem[SIG97]{SIG}
 Robert Skelton, Tetsuya Iwasaki and Karolos Grigoriadis,
{\it  A Unified Algebraic Approach to Linear Control Design,}
Taylor \& Francis, 1997. 


\bibitem[Pop10]{Pop} Gelu Popescu:
\textit{Free holomorphic automorphisms of the unit ball of $B(H)^n$},
{J. reine angew. Math.} {638} (2010) 119--168.




\bibitem[WSV12]{WSV}
 Henry Wolkowicz, Romesh Saigal and  Lieven Vandenberghe (eds.),
{\it  Handbook of semidefinite programming: theory,
algorithms, and applications,}
  International Series in Operations Research \& Management Science 27,
Springer, 2012

\end{thebibliography}
\end{document}